\documentclass[a4paper,11pt]{article}

\usepackage[applemac]{inputenc}
\usepackage{enumerate}
\usepackage{lmodern}
\usepackage[T1]{fontenc}
\usepackage{verbatim}
\usepackage{textcomp}
\usepackage[english]{babel}
\usepackage[a4paper,vmargin={3.5cm,3.5cm},hmargin={2.5cm,2.5cm}]{geometry}
\usepackage[font=sf, labelfont={sf,bf}, margin=1cm]{caption}
\usepackage[pagebackref=false,
pdftex]{hyperref}
\usepackage[pdftex]{color,graphicx}

\usepackage{amsmath,amsfonts,amssymb,amsthm,mathrsfs}
\usepackage{colonequals} 
\usepackage{stmaryrd}     

\usepackage{calc}
\usepackage{graphicx}
\makeatletter
\newlength{\temp@wc@width}
\newlength{\temp@wc@height}
\newcommand{\widecheck}[1]{%
  \setlength{\temp@wc@width}{\widthof{$#1$}}%
  \setlength{\temp@wc@height}{\heightof{$#1$}}%
  #1\hspace{-\temp@wc@width}%
  \raisebox{\temp@wc@height+1pt}[\heightof{$\widehat{#1}$}]%
     {\rotatebox[origin=c]{180}{\vbox to 0pt{\hbox{$\widehat{\hphantom{#1}}$}}}}%
}
\makeatother

\usepackage{cleveref}
  \crefname{theorem}{Theorem}{Theorems}
  \crefname{lemma}{Lemma}{Lemmas}
  \crefname{remark}{Remark}{Remarks}
  \crefname{proposition}{Proposition}{Propositions}
  \crefname{definition}{Definition}{Definitions}
  \crefname{corollary}{Corollary}{Corollaries}
  \crefname{section}{Section}{Sections}
  \crefname{figure}{Figure}{Figures}

\usepackage{color}

\newtheorem{theorem}{Theorem}[]

\newtheorem{proposition}[theorem]{Proposition}
\newtheorem{lemma}[theorem]{Lemma}
\newtheorem{corollary}[theorem]{Corollary}
\theoremstyle{definition}

\def\a{^{(\alpha)}}
\def\al{\alpha}
\def\lal{\lambda_\alpha}

\def\bd{{\bf d}}
\def\rd{{\mathrm d}}
\def\cc{{\mathcal C}}
\def\cw{{\mathcal W}}

\def\t{{\mathcal T}}

\def\v{{\mathcal V}}

\def\wcc{\widehat{\mathcal{C}}}
\def\wt{\widetilde}

\def\R{{\mathbb R}}
\def\P{{\mathbb P}}
\def\E{{\mathbb E}}
\def\N{{\mathbb N}}
\def\Z{{\mathbb Z}}

\def\build#1_#2^#3{\mathrel{\mathop{\kern 0pt#1}\limits_{#2}^{#3}}}
\def\rem{\noindent{\bf Remark. }}

\title{Typical behavior of the harmonic measure in critical Galton--Watson trees with infinite variance offspring distribution}

\author{Shen LIN 
\thanks{Supported in part by the grant ANR-14-CE25-0014 (ANR GRAAL)} \\
\small \it LPMA, Universit\'e Pierre et Marie Curie, Paris, France \\
\small \textit{E-mail}: \texttt{shen.lin.math@gmail.com} }
\date{}

\begin{document}

\maketitle

\begin{abstract}
We study the typical behavior of the harmonic measure in large critical Galton--Watson trees whose offspring distribution is in the domain of attraction of a stable distribution with index $\alpha\in (1,2]$. 
Let $\mu_n$ denote the hitting distribution of height $n$ by simple random walk on the critical Galton--Watson tree conditioned on non-extinction at generation~$n$. 
We extend the results of~\cite{LIN2} to prove that, with high probability, the mass of the harmonic measure $\mu_n$ carried by a random vertex uniformly chosen from height $n$ is approximately equal to $n^{-\lambda_\alpha}$, where the constant $\lambda_\alpha >\frac{1}{\alpha-1}$ depends only on the index $\alpha$. 
In the analogous continuous model, this constant $\lambda_\alpha$ turns out to be the typical local dimension of the continuous harmonic measure. 
Using an explicit formula for $\lambda_\alpha$, we are able to show that $\lambda_\alpha$ decreases with respect to $\alpha\in(1,2]$, and it goes to infinity at the same speed as $(\alpha-1)^{-2}$ when $\alpha$ approaches 1.

\medskip
\noindent {\bf Keywords.} size-biased Galton--Watson tree, reduced tree, harmonic measure, uniform measure, simple random walk and Brownian motion on trees.

\smallskip
\noindent{\bf AMS 2010 Classification Numbers.} 60J80, 60G50, 60K37. 
\end{abstract}

\section{Introduction}
In this paper, we continue our previous work~\cite{LIN2} by extending its results to the critical Galton--Watson trees whose offspring distribution may have infinite variance. 
To be more precise, let $\rho$ be a non-degenerate probability measure on $\Z_{+}$ with mean one, and we assume throughout this paper that $\rho$ is in the domain of attraction of a stable distribution of index $\alpha\in(1,2]$, which means that
\begin{equation}
\label{eq:stable-attraction}
\sum\limits_{k\geq 0}\rho(k)r^{k}= r+(1-r)^{\alpha}L(1-r)\qquad \mbox{ for any } r\in [0,1),
\end{equation} 
where the function $L(x)$ is slowly varying as $x \to 0^{+}$. The finite variance condition for $\rho$ is sufficient for the previous statement to hold with $\alpha=2$. When $\alpha\in (1,2)$, by classical results of~\cite[Chapters XIII and XVII]{F71}, the condition (\ref{eq:stable-attraction}) is satisfied if and only if the tail probability 
\begin{displaymath}
\sum\limits_{k\geq x} \rho(k)=\rho([x,+\infty))
\end{displaymath}
varies regularly with exponent $-\alpha$ as $x \to +\infty$. 
 
Under the probability measure $\P$, for every integer $n\geq 0$, we let $\mathsf{T}^{(n)}$ be a Galton--Watson tree with offspring distribution $\rho$, conditioned on non-extinction at generation $n$. Conditionally given the tree $\mathsf{T}^{(n)}$, we consider a simple random walk on $\mathsf{T}^{(n)}$ starting from the root. The probability distribution of its first hitting point of generation $n$ will be called the harmonic measure $\mu_{n}$, which is supported on the set $\mathsf{T}^{(n)}_{n}$ consisting of all vertices of $\mathsf{T}^{(n)}$ at generation~$n$. 
 
Let $q_{n}>0$ be the probability that a critical Galton--Watson tree $\mathsf{T}^{(0)}$ survives up to generation~$n$. It has been shown by Slack~\cite{S68} that, as $n\to \infty$, the probability $q_{n}$ decreases as $n^{-\frac{1}{\alpha-1}}$ up to multiplication by a slowly varying function, and $q_{n}\#\mathsf{T}^{(n)}_{n}$ converges in distribution to a non-trivial limit distribution on $\R_{+}$ that we will specify shortly. The main result of the present work generalizes Theorem 1 of~\cite{LIN2} from the finite variance case (with $\alpha=2$) to all $\alpha \in (1,2]$.

\begin{theorem}
\label{thm:dim-discrete}
Let $\Omega_n$ be a random vertex uniformly chosen from $\mathsf{T}^{(n)}_n$. If the offspring distribution $\rho$ has mean one and belongs to the domain of attraction of a stable distribution of index $\alpha\in (1,2]$, there exists a constant $\lal > \frac{1}{\alpha-1}$, which only depends on $\alpha$, such that for every $\delta>0$, we have 
\begin{equation}
\label{eq:dim-discrete} 
\lim_{n\to \infty}\P \Big( n^{-\lal-\delta} \leq \mu_n(\Omega_n) \leq n^{-\lal+\delta} \Big)= 1.
\end{equation}
\end{theorem}

Roughly speaking, the previous theorem asserts that if we look at a typical vertex at level $n$ of the conditional Galton--Watson tree $\mathsf{T}^{(n)}$, then this random vertex carries with high probability a mass of order $n^{-\lal}$ given by the harmonic measure $\mu_n$. 
Since $\lal> \frac{1}{\alpha-1}$, we see particularly that the harmonic measure on the conditional Galton--Watson tree is not uniformly spread and exhibits a fractal behavior. 

The fractal nature of the harmonic measure in critical Galton--Watson tree has been initially investigated by Curien and Le Gall~\cite{CLG13} from another perspective.
To compare with Theorem~\ref{thm:dim-discrete}, their main result in the finite variance case has also been generalized to the present setting as Theorem 1.1 in~\cite{LIN1} : under the same assumption of Theorem~\ref{thm:dim-discrete}, we have the existence of another constant $\beta_\al \in(0,\frac{1}{\alpha-1})$ only depending on $\alpha$, such that for every $\delta>0$, we have the convergence in $\mathbb{P}$-probability
\begin{equation}
\mu_n\Big(\big\{v\in \mathsf{T}^{(n)}_n\colon n^{-\beta_{\alpha}-\delta}\leq \mu_n(v) \leq n^{-\beta_{\alpha}+\delta}\big\}\Big)  \xrightarrow[n\to\infty]{(\mathbb{P})} 1\,.  
\label{eq:theorem1-LIN1}
\end{equation}
In other words, with high probability the mass given by the harmonic measure $\mu_n$ to a random vertex at level $n$ drawn with respect to the same harmonic measure is of order $n^{-\beta_\al}$ with $\beta_\al<\frac{1}{\alpha-1}$. 
We can thus say that the harmonic measure $\mu_n$ is mainly supported on a subset of size approximately equal to $n^{\beta_\al}$, which is much smaller than the whole size of $\mathsf{T}^{(n)}_n$. 
Notice that the family $(\beta_\al,1<\al\leq 2)$ is shown in~\cite[Theorem 1.3]{LIN1} to be bounded from above in $\R_+$.

As the hitting distribution $\mu_{n}$ of generation $n$ by simple random walk on $\mathsf{T}^{(n)}$ is unaffected if we remove the branches of $\mathsf{T}^{(n)}$ that do not reach height $n$, we can simplify the model by considering merely simple random walk on the reduced tree $\mathsf{T}^{*n}$ associated with $\mathsf{T}^{(n)}$, which consists of all vertices of $\mathsf{T}^{(n)}$ that have at least one descendant at generation $n$. 

When the critical offspring distribution $\rho$ has infinite variance, the study of scaling limits of $\mathsf{T}^{*n}$ goes back to Vatutin~\cite{V77} and Yakymiv~\cite{Y80}. 
If we scale the graph distance by the factor $n^{-1}$, as $n\to \infty$ the discrete reduced tree $n^{-1}\mathsf{T}^{*n}$ converges to a random continuous tree $\Delta^{(\alpha)}$ that we now describe. 
For every $\alpha\in(1,2]$, we define the $\alpha$-offspring distribution $\theta_{\alpha}$ as follows. 
If $\alpha=2$, we let $\theta_{2}$ be the Dirac measure at~2. If $\alpha<2$, $\theta_{\alpha}$ is the probability measure on $\mathbb{Z}_{+}$ given by
\begin{eqnarray*}
\theta_{\alpha}(0) &= &\theta_{\alpha}(1)\; = \;0, \\
\theta_{\alpha}(k) &= &\frac{\alpha\, \Gamma(k-\alpha)}{k!\,\Gamma(2-\alpha)}\,=\,\frac{\alpha(2-\alpha)(3-\alpha)\cdots(k-1-\alpha)}{k!}\,, \quad \forall k\geq 2,
\end{eqnarray*}
where $\Gamma(\cdot)$ denotes the usual Gamma function of Euler. The mean of $\theta_{\alpha}$ is equal to $\frac{\al}{\al-1}$. 
To construct the random tree $\Delta\a$, we begin with the genealogical tree $\Delta_0\a$ of a family of particles that evolves in continuous time according to the following rules. 
At time 0 there is a single particle, this particle lives for a random time uniformly distributed over $[0,1]$, then dies and gives birth to a random number of new particles. 
This number of offspring is distributed according to~$\theta_\al$ and is assumed to be independent of the lifetime of the initial particle. 
Inductively, all the particles evolve independently of each other, and each new particle appeared at time $t\in (0,1)$ dies at a time uniformly distributed over $[t,1]$ and gives birth to an independent number of new particles according to the supercritical offspring distribution $\theta_\al$. 
After taking completion of $\Delta_0\a$ with respect to its natural intrinsic metric $\mathbf{d}$, we obtain a random compact rooted $\R$-tree $\Delta^{(\alpha)}$ that will be called the reduced stable tree of parameter $\alpha$.
Its boundary $\partial \Delta^{(\alpha)}$ is composed of all points of $\Delta\a$ at height 1.
We refer to Section~\ref{sec:treedelta} for a precise definition of $\Delta\a$. 

The description above makes clear the recursive structure of $\Delta\a$: 
let $U$ be a uniform random variable over $[0,1]$ and let $N_\al\in\N$ be a random variable of law $\theta_\al$.
We take $(\Delta\a_i)_{i\geq 1}$ to be a sequence of independent copies of $\Delta\a$, and we assume the independence between $U, N_\al$ and $\Delta\a_i, i\geq 1$.
If we attach to the top of a single line segment of length $U$ the roots of the rescaled trees $(1-U)\Delta\a_1, \ldots, (1-U)\Delta\a_{N_\al}$, the resulting tree rooted at the origin of the initial line segment will have the same distribution as $\Delta_\al$. 
See Fig.~\ref{fig:recur} for an illustration. 

\begin{figure}[!h]
 \begin{center}
 \includegraphics[width=9cm]{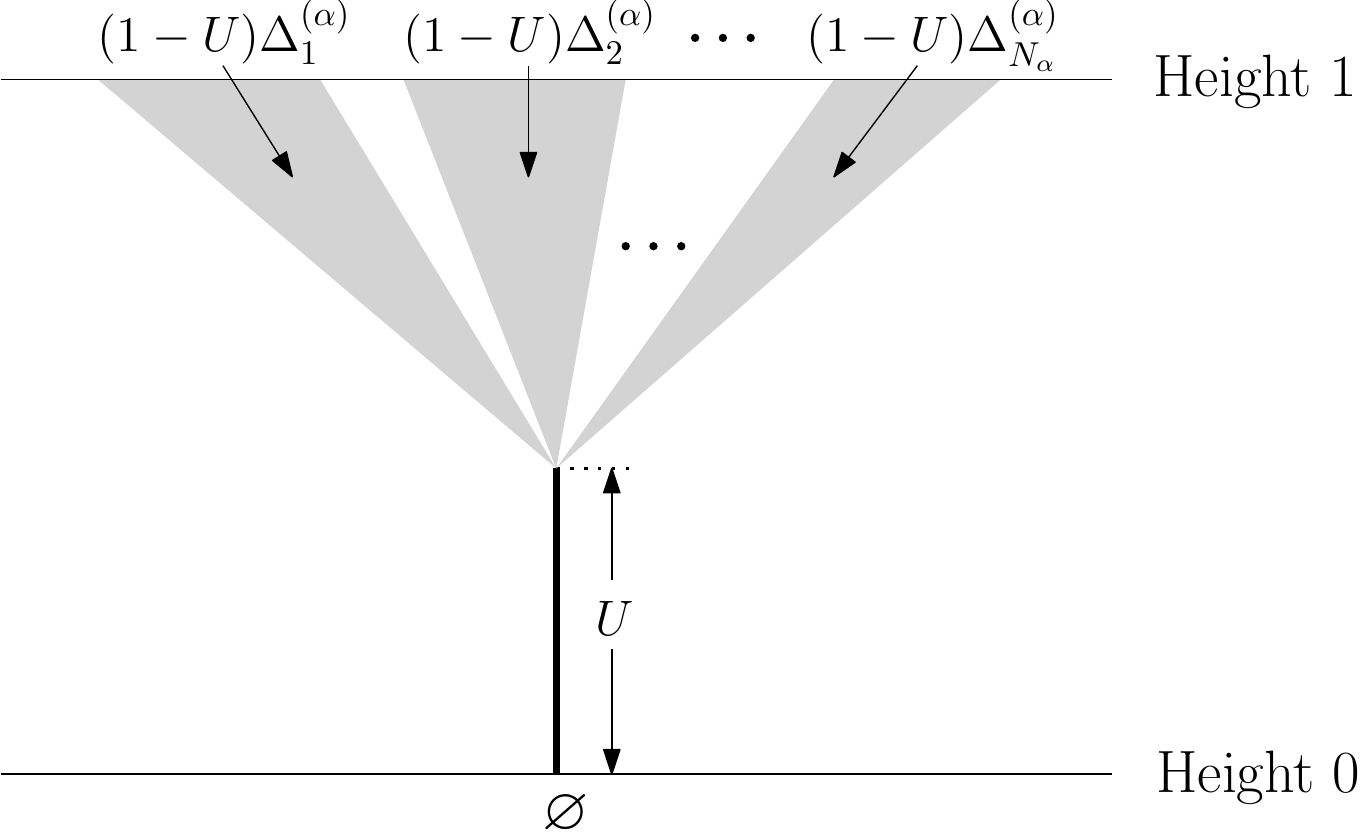}
 \caption{\label{fig:recur}Recursive structure of the reduced stable tree $\Delta\a$}
 \end{center}
\end{figure}

As the continuous analog of simple random walk, Brownian motion on $\Delta\a$ starting from the root can be defined up until its first hitting time of $\partial \Delta\a$. It behaves like linear Brownian motion as long as it stays inside a line segment of $\Delta\a$. It is reflected at the root of $\Delta\a$ and when it arrives at a branching point, it chooses one of the adjacent line segments with equal probabilities. 
We define the (continuous) harmonic measure $\mu_{\alpha}$ on $\Delta\a$ as the (quenched) distribution of the first hitting point of $\partial \Delta\a$ by Brownian motion.

The behavior of $\mu_{\alpha}$ is closely related to that of the discrete harmonic measure $\mu_n$. 
In order to state a result analogous with Theorem~\ref{thm:dim-discrete} in the continuous setting, we must first make sense of a ``typical'' point chosen from the boundary $\partial \Delta\a$. 
To this end, we introduce another (non-compact) random rooted $\R$-tree $\Gamma\a$ endowed with the same branching structure as $\Delta\a$, such that each point of $\Gamma\a$ at height $y\in [0,\infty)$ corresponds bijectively to a point of $\Delta_0\a$ at height $1-e^{-y}\in [0,1)$. 
It is easy to verify that the resulting new tree $\Gamma\a$ is a continuous-time Galton--Watson tree of branching rate 1 with the supercritical offspring distribution $\theta_\al$. In particular, $\Gamma^{(2)}$ is the Yule tree which describes the genealogy of the classical Yule process. 
By definition, the boundary $\partial \Gamma\a$ of $\Gamma\a$ is the set of all infinite geodesics in $\Gamma\a$ starting from the root (these will be called geodesic rays). 
Since $\Delta\a$ and $\Gamma\a$ share the same branching structure, both $\partial \Delta\a$ and $\partial \Gamma\a$ can be canonically identified with a common random subset of $\N^{\N}$. 

For every $r>0$, we write $\Gamma\a_r$ for the level set of $\Gamma\a$ at height $r$, and we know that $\E[\#\Gamma_r^{(\alpha)}]= \exp (\frac{r}{\alpha-1})$. 
By a martingale argument, one can define
\begin{displaymath}
\mathcal{W}\a \colonequals \lim_{r\to \infty} e^{-\frac{r}{\al-1}}\# \Gamma_r\a.
\end{displaymath}
As $\sum \theta_\al(k)k\log k <\infty$, the Kesten--Stigum theorem (for continuous-time Galton--Watson trees, see e.g.~\cite[Theorem III.7.2]{AN}) implies that the previous convergence holds in the $L^1$-sense, $\E[\cw\a]=1$, and $\cw\a>0$ almost surely. It follows from Theorem III.8.3 in~\cite{AN} that 
\begin{displaymath}
\E \Big[e^{-u\mathcal{W}\a}\Big] = 1-\frac{u}{(1+u^{\alpha-1})^{\frac{1}{\alpha-1}}} \quad \mbox{ for any } u\in (0,\infty).
\end{displaymath}
Moreover, according to Theorem 1 of~\cite{S68}, $q_{n}\#\mathsf{T}^{(n)}_{n}$ converges in distribution to this martingale limit $\mathcal{W}\a$ as $n\to \infty$.
For every $x\in \Gamma\a$, we let $H(x)$ denote the height of $x$ in $\Gamma\a$, and we write $\Gamma\a[x]$ for the tree of descendants of $x$ in $\Gamma\a$, viewed as an infinite random $\R$-tree rooted at $x$. For every $r>0$, we write $\Gamma\a_r[x]$ for the level set at height $r$ of the tree $\Gamma\a[x]$. If one thinks of $\Gamma\a[x]$ as a subtree of $\Gamma\a$, the set $\Gamma\a_r[x]$ consists of all the points of $\Gamma\a$ at height $r+H(x)$ that are descendants of $x$. We similarly define
\begin{displaymath}
\mathcal{W}\a[x] \colonequals \lim_{r\to \infty} e^{-\frac{r}{\al-1}}\# \Gamma\a_r[x],
\end{displaymath}
which has the same distribution as $\cw\a$. 
The uniform measure $\bar \omega_\al$ on $\partial \Gamma\a$ is defined as the unique probability measure on $\partial \Gamma\a$ satisfying that, for every $x\in \Gamma\a$ and for every geodesic ray $\mathbf{v}\in \partial \Gamma\a$ passing through $x$,
\begin{displaymath}
\bar \omega_\al(\mathcal{B}(\mathbf{v},H(x)))=  \exp \Big(-\frac{H(x)}{\al-1} \Big)\frac{\mathcal{W}\a[x]}{\mathcal{W}\a},
\end{displaymath}
where $\mathcal{B}(\mathbf{v},H(x))$ stands for the set of all geodesic rays in $\Gamma\a$ that coincide with $\mathbf{v}$ up to height $H(x)$. In existing literature, we also call $\bar \omega_\al$ the branching measure on the boundary of $\Gamma\a$. 
Since $\partial \Delta\a$ can be identified with $\partial \Gamma\a$ as explained above, we let $\omega_\al$ be the (random) probability measure on $\partial \Delta\a$ induced by $\bar \omega_\al$, which will be referred to as the uniform measure on $\partial \Delta\a$.
The following result is an extension of Theorem 2 in~\cite{LIN2}.

\begin{theorem}
\label{thm:dim-continu}
For every $\al\in (1,2]$, with the same constant $\lambda_\al$ as in Theorem~\ref{thm:dim-discrete}, we have $\P$-a.s.~$\omega_\al(\mathrm{d}\mathbf{v})$-a.e.
\begin{eqnarray}
\lim_{r\downarrow 0} \frac{\log \mu_\al (\mathcal{B}_{\bd}(\mathbf{v},r))}{\log r} &= &\lambda_\al \,, \label{eq:loc-dim-harm}\\
\lim_{r\downarrow 0} \frac{\log \omega_\al (\mathcal{B}_{\bd}(\mathbf{v},r))}{\log r} &=  & \frac{1}{\al-1} \,,\label{eq:loc-dim-unif}
\end{eqnarray}
where $\mathcal{B}_{\bd}(\mathbf{v},r)$ stands for the closed ball of radius $r$ centered at $\mathbf{v}$ in the metric space $(\Delta\a,\bd)$.
\end{theorem}

By Duquesne and Le Gall~\cite[Theorem 5.5]{DLG05}, the Hausdorff measure of $\partial \Delta\a$ with respect to $\mathbf{d}$ is $\P$-a.s.~equal to $\frac{1}{\al-1}$. According to Lemma 4.1 in~\cite{LPP95}, assertion~\eqref{eq:loc-dim-unif} of the preceding theorem implies that $\P$-a.s.~the uniform measure $\omega_\al$ has the same Hausdorff dimension as the whole boundary of the reduced stable tree. 
Meanwhile, taking account of \eqref{eq:loc-dim-harm}, we may interpret $\lambda_\al$ as the local dimension of the harmonic measure $\mu_\al$ at a typical point of the boundary $\partial \Delta\a$. 
In contrast, the Hausdorff dimension of the harmonic measure $\mu_\al$ is $\P$-a.s.~equal to the constant $\beta_\al$ appearing in~(\ref{eq:theorem1-LIN1}), since Theorem 1.2 of~\cite{LIN1} shows that $\P$-a.s.~$\mu_\al(\mathrm{d}\mathbf{v})$-a.e.,
\begin{displaymath}
\lim_{r\downarrow 0} \frac{\log \mu_\al (\mathcal{B}_{\bd}(\mathbf{v},r))}{\log r} = \beta_\al.
\end{displaymath}
Since $\beta_\al<\frac{1}{\al-1}<\lambda_\al$, $\P$-a.s.~the set 
\begin{displaymath}
B=\big\{\mathbf{v}\in \partial \Delta\a \colon \lim_{r\downarrow 0} \frac{\log \mu_\al (\mathcal{B}_{\bd}(\mathbf{v},r))}{\log r} = \beta_\al \big\}
\end{displaymath}
satisfies that $\mu_\al(B)=1$ whereas $\omega_\al(B)=0$. 

\begin{corollary}
\label{corol-intro}
For every $\al\in (1,2]$, $\P$-a.s.~the two measures $\mu_\al$ and $\omega_\al$ on the boundary of $\Delta\a$ are mutually singular.
\end{corollary}

When $\alpha$ approaches 1, the Hausdorff dimension $\beta_\al$ of the harmonic measure $\mu_\al$ remains bounded (see Theorem 1.3 in~\cite{LIN1}). 
Naively, it would seem that the typical local dimension $\lal$ of $\mu_\al$ would behave asymptotically like $\frac{C}{\al-1}$ for some positive constant $C$. 
However, the explosion of $\lal$ turns out to be much faster than $(\al-1)^{-1}$ when $\alpha$ decreases to 1.

\begin{proposition}
\label{prop:dimension-explosion}
We have that
$$ 0 < \liminf_{\alpha \downarrow 1} \,(\al-1)^2\lal \leq \limsup_{\alpha \downarrow 1} \,(\al-1)^2\lal <\infty.$$
\end{proposition}

Our proof of Proposition~\ref{prop:dimension-explosion} relies on the fact that the constant $\lal$ appearing in Theorems~\ref{thm:dim-discrete} and~\ref{thm:dim-continu} can be expressed explicitly with help of the conductance of $\Delta\a$. 
Informally, if we think of the random tree $\Delta\a$ as an electric network of resistors with unit resistance per unit length, the effective conductance $\mathcal{C}\a$ between the root and the boundary $\partial \Delta\a$ is a random variable larger than~1. From a probabilistic point of view, it is the mass under the Brownian excursion measure for the excursion paths away from the root that hit height 1. 
Following the definition of $\Delta\a$ and its electric network interpretation, the distribution of $\mathcal{C}\a$ satisfies the recursive distributional equation
\begin{equation}
\label{eq:rde}
\mathcal{C}^{(\alpha)}\,\overset{(\mathrm{d})}{=\joinrel=}\, \bigg (U + \frac{1-U}{ \mathcal{C}\a_1+ \mathcal{C}\a_2+\cdots+\mathcal{C}\a_{N_{\alpha}}}\bigg)^{-1},
\end{equation} 
where $(\mathcal{C}\a_{i})_{i\geq1}$ are i.i.d.~copies of $\mathcal{C}\a$, the integer-valued random variable $N_{\alpha}$ is distributed according to $\theta_{\alpha}$, and $U$ is uniformly distributed over $[0,1]$. All these random variables are supposed to be independent. 

\begin{proposition}
\label{prop:dim-formula}
For any $\alpha \in (1,2]$, the constant $\lal$ appearing in Theorems \ref{thm:dim-discrete} and~\ref{thm:dim-continu} is given by 
\begin{equation}
\label{eq:lal-value}
\lal= \E\big[\cw \a \cc \a \big]-1.
\end{equation}
Moreover, $\lal$ is decreasing for all $\al\in (1,2]$. 
\end{proposition}

The surprisingly simple formula~(\ref{eq:lal-value}) provides a unified expression of $\lal$ in terms of $\cw\a$ and $\cc\a$. 
We will see in Section~\ref{sec:conductance} that the product $\cw \a \cc \a$ has the same first moment as the conductance $\widehat \cc \a$ of a size-biased version of the reduced stable tree $\Delta\a$.

Note that for $1<\al \leq 2$, the $\al$-offspring distribution $\theta_\al$ is decreasing for the usual stochastic partial order. 
This property allows one to construct simultaneously all reduced stable trees $\Delta\a, \al\in (1,2]$ as a nested family, so that $\Delta^{(\al_2)} \subseteq \Delta^{(\al_1)}$ for all $1<\al_1\leq \al_2\leq 2$ (see Section~2.4 in~\cite{LIN1}). 
While the monotonicity of the typical local dimension $\lal$ is affirmed by the previous result, one question still open is whether the Hausdorff dimension $\beta_\al$ of the continuous harmonic measure $\mu_\al$ is also decreasing with respect to $\al$.

The rest of this paper is divided into three parts. 
The continuous model of Brownian motion on $\Delta\a$ is studied in Section~\ref{sec:continuous}, where we prove Theorem~\ref{thm:dim-continu}, Propositions~\ref{prop:dimension-explosion} and~\ref{prop:dim-formula}. Then we set up notation and terminology for the discrete setting in Section~\ref{sec:discrete}, while Section~\ref{sec:proof-thm1} is devoted to proving Theorem~\ref{thm:dim-discrete}. 
The arguments are basically adapted from those used in the finite variance case. 
However, we emphasize that several modifications are indispensable to circumvent the problem of infinite variance. 
Instead of repeating some highly analogous reasoning, we refer the reader to \cite{LIN2} for details that we omit. 

\smallskip
{\bf Acknowledgments.} The author is grateful to an anonymous referee for many comments that greatly improved the paper. 

\section{The continuous setting}
\label{sec:continuous}

\subsection{The reduced stable tree}
\label{sec:treedelta}

Let us begin with a formal definition of the reduced stable tree ${\Delta^{(\alpha)}}$ of parameter $\alpha \in(1,2]$. We set 
$$\mathcal{V} = \bigcup_{n=0}^\infty \mathbb{N}^n\,,$$
where by convention $\mathbb{N}=\{1,2,\ldots\}$ and $\mathbb{N}^0=\{\varnothing\}$. If $v=(v_1,\ldots,v_n)\in\mathcal{V}$, we set~$|v|=n$ the generation of the vertex $v$ (in particular, $|\varnothing|=0$),  and if $n\geq 1$,
we define the parent 
of~$v$ as $\bar v=(v_1,\ldots,v_{n-1})$ and then say that $v$ is a child of $\bar v$. For two elements $v=(v_1,\ldots,v_n)$ and $v'=(v'_1,\ldots,v'_m)$ belonging to $\mathcal{V}$, their concatenation is $vv'\colonequals (v_1,\ldots,v_n,v'_1,\ldots,v'_m)$. The notions of a descendant and an ancestor of an element of $\v$ are defined in the obvious way, with the convention that every $v\in \v$ is both an ancestor and a descendant of itself. 

An infinite subset $\Pi$ of $\mathcal{V}$ is an infinite discrete (rooted ordered) tree without leaves if there exists a collection of positive integers $k_{v}=k_{v}(\Pi)\in \N$ for every $v\in \mathcal{V}$ such that 
\begin{displaymath}
\Pi=\{\varnothing\}\cup\{(v_{1},\ldots,v_{n})\in \mathcal{V}: v_{j}\leq k_{(v_{1},\ldots,v_{j-1})} \mbox{ for every } 1\leq j\leq n\}.
\end{displaymath}

Recall that the generating function of the $\alpha$-offspring distribution $\theta_{\alpha}$ 
is given (see for instance~\cite[p.74]{DLG02}) as 
\begin{equation}
\label{eq:gene-fct}
\sum\limits_{k\geq 0}^{}\theta_{\alpha}(k)\,r^k =\frac{(1-r)^{\alpha}-1+\alpha r}{\alpha-1},\quad \forall r\in (0,1].
\end{equation}

For fixed $\alpha\in (1,2]$, we introduce a collection $(K_{\alpha}(v))_{v\in\v}$ of independent random variables distributed according to $\theta_{\alpha}$ under the probability measure $\P$, and define a random infinite discrete tree 
\begin{displaymath}
\Pi^{(\alpha)}\colonequals \{\varnothing\}\cup\{(v_{1},\ldots,v_{n})\in \v\colon v_{j}\leq K_{\alpha}((v_{1},\ldots,v_{j-1})) \mbox{ for every }1\leq j\leq n\}\,.
\end{displaymath}
We point out that $\Pi^{(\alpha)}$ is just a supercritical Galton--Watson tree with offspring distribution~$\theta_\al$. In particular, $\Pi^{(2)}$ is an infinite binary tree.

Let $(U_v)_{v\in\v}$ be another collection, independent of $(K_{\alpha}(v))_{v\in\v}$, consisting of independent real random variables uniformly distributed over $[0,1]$ under the same probability measure~$\P$. We set now $Y_\varnothing=U_\varnothing$ and then by induction, for every $v\in \Pi^{(\alpha)}$ different from the root, $Y_v = Y_{\bar v} + U_v(1- Y_{\bar v})$. 
Note that a.s.~$0\leq Y_v< 1$ for every $v\in \Pi^{(\alpha)}$. Consider the set
$$\Delta^{(\alpha)}_{0} \colonequals \big(\{\varnothing\}\times [0,Y_\varnothing] \big)\cup  \bigg(\bigcup_{v\in\Pi^{(\alpha)}\backslash\{\varnothing\}} \{v\} \times (Y_{\bar v}, Y_v]\bigg).$$
There is a straightforward way to define a metric $\bd$ on $\Delta^{(\alpha)}_{0}$, so that 
$(\Delta^{(\alpha)}_{0},\bd)$ is a (non-compact) $\R$-tree and, for every $x=(v,r)\in \Delta^{(\alpha)}_{0}$, we have $\bd((\varnothing,0), x)=r$. 
See \cref{fig:Delta} for an illustration of the tree $\Delta^{(\alpha)}_{0}$ when $\alpha <2$.

\begin{figure}[!h]
 \begin{center}
 \includegraphics[width=12cm]{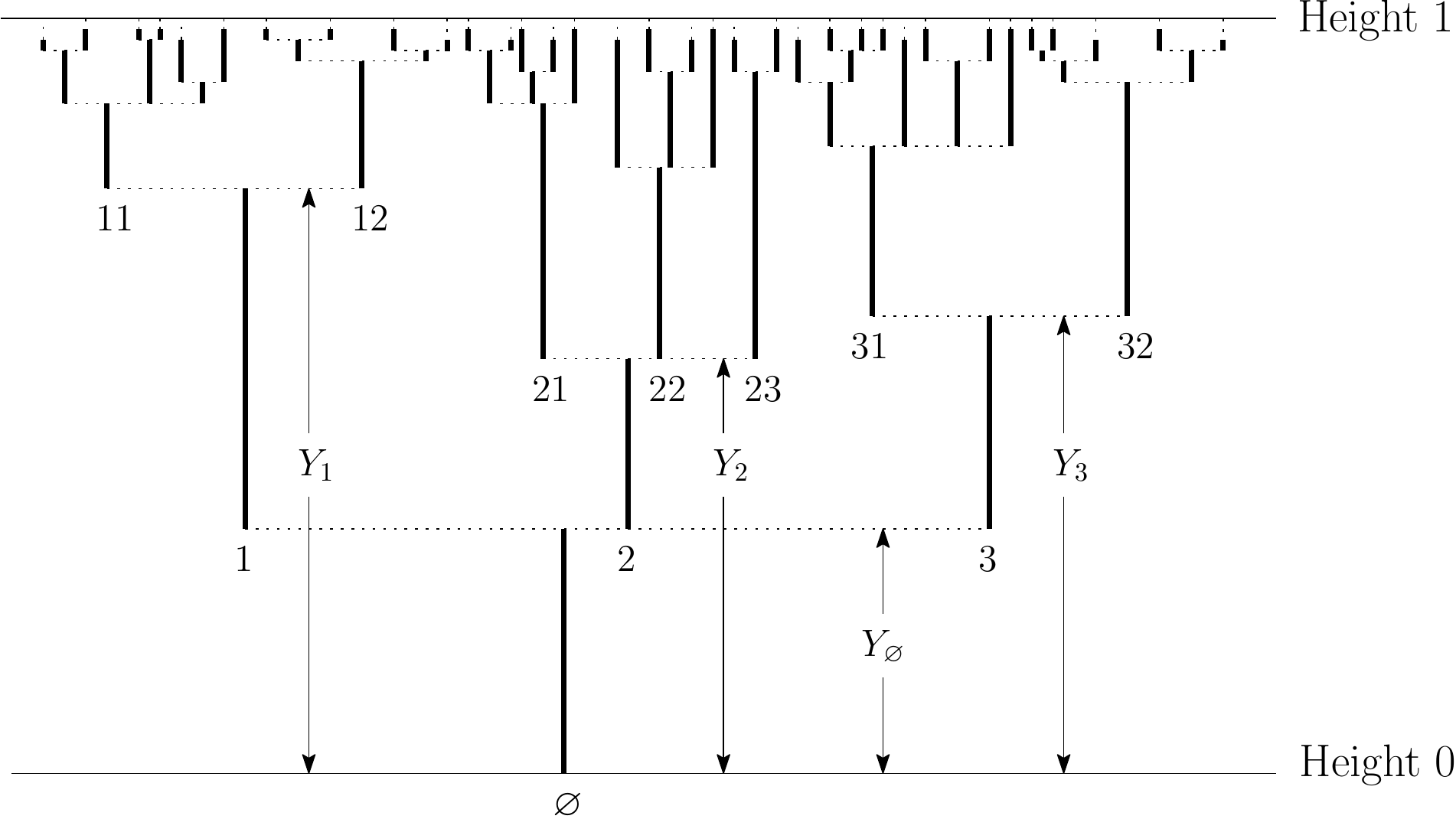}
 \caption{\label{fig:Delta}The random tree $\Delta^{(\alpha)}_{0}$ when $1\leq \alpha<2$}
 \end{center}
 \end{figure}

The reduced stable tree $\Delta\a$ is the completion of $\Delta^{(\alpha)}_{0}$ with respect to the metric $\bd$. It is immediate to see that $(\Delta^{(\alpha)},\bd)$ is a compact $\R$-tree, and 
$$\Delta^{(\alpha)}=\Delta^{(\alpha)}_{0} \cup \partial \Delta^{(\alpha)}$$
where $\partial \Delta^{(\alpha)}\colonequals \{x\in\Delta^{(\alpha)}\colon \bd((\varnothing,0), x)=1\}$ is the boundary of $\Delta\a$, which can be identified with a random subset of $\N^{\N}$. 
The point $(\varnothing,0)$ is called the root of $\Delta^{(\alpha)}$. For every $x\in\Delta^{(\alpha)}$, we set its height $H(x)=\bd((\varnothing,0), x)$. 
A genealogical order on $\Delta^{(\alpha)}$ can be defined by setting $x\prec y$ if and only if $x$ belongs to the geodesic path from the root to $y$.

Conditionally on $\Delta^{(\alpha)}$, we can define Brownian motion $(B_t)_{t\geq 0}$ on $\Delta^{(\alpha)}$ starting from the root and up to its first hitting time $T$ of $\partial \Delta^{(\alpha)}$. 
See e.g.~\cite[Section 2.1]{CLG13} for the details of this construction. 
The harmonic measure $\mu_{\alpha}$ is the distribution of $B_{T-}$, which is a (random) probability measure on $\partial \Delta^{(\alpha)}\subseteq \N^\N$.

\subsection{The continuous-time Galton--Watson tree}
\label{sec:ctgwtree}

To introduce a new tree sharing the same branching structure as $\Delta^{(\alpha)}$, we start with the same infinite discrete tree $\Pi^{(\alpha)}$ introduced in Section~\ref{sec:treedelta}. 
Consider now a collection $(V_v)_{v\in\mathcal{V}}$ of independent real random variables exponentially distributed with mean $1$ under the probability measure $\P$. We set $Z_\varnothing=V_\varnothing$ and then by induction, for every $v\in \Pi^{(\alpha)}$ different from the root, $Z_v = Z_{\bar v} + V_v$.
The continuous-time Galton--Watson tree (hereafter to be called CTGW tree for short) of stable index~$\alpha$ is the set
$$\Gamma^{(\alpha)} \colonequals\big(\{\varnothing\}\times [0,Z_\varnothing]\big) \cup  \bigg(\bigcup_{v\in \Pi^{(\alpha)}\backslash\{\varnothing\}} \{v\} \times (Z_{\bar v}, Z_v]\bigg),$$
which is equipped with the metric $d$ defined in the same way as $\bd$ in the preceding subsection. 
For this metric, $\Gamma^{(\alpha)}$ is a.s.~an infinite $\R$-tree. For every $x=(v,r)\in\Gamma^{(\alpha)}$, we keep the notation $H(x)=r=d((\varnothing,0),x)$ for the height of $x$.  

Observe that if $U$ is uniformly distributed over $[0,1]$, the random variable $-\log(1-U)$ is exponentially distributed with mean $1$. Hence we may and will suppose that the collection $(V_v)_{v\in\mathcal{V}}$ is constructed from the collection $(U_v)_{v\in\mathcal{V}}$ in the previous subsection via the formula $V_v=-\log(1-U_v)$ for every $v\in\mathcal{V}$. 
A homeomorphism from $\Delta^{(\alpha)}_{0}$ onto $\Gamma^{(\alpha)}$ is given by the mapping $\Psi$ defined as $\Psi(v,r)\colonequals (v,-\log(1-r))$ for every $(v,r)\in\Delta^{(\alpha)}_{0}$.

By stochastic analysis, we can write for every $t\in[0,T)$,
\begin{equation}
\label{BM-CTGW}
\Psi(B_t) =W\Big(\int_0^t (1-H(B_s))^{-2}\,\mathrm{d}s\Big)
\end{equation}
where $(W(t))_{t\geq 0}$ is Brownian motion with constant drift $1/2$ towards infinity on the CTGW tree~$\Gamma\a$ (this process is defined in a similar way as Brownian motion on $\Delta^{(\alpha)}$, except that it behaves like Brownian motion with drift $1/2$ on every line segment of the tree). 
It is easy to see that the Brownian motion $W$ is transient. 
From now on, when we speak about Brownian motion on the CTGW tree or on other similar infinite trees, we will always mean Brownian motion with drift $1/2$ towards infinity. 

By definition, the boundary of $\Gamma^{(\alpha)}$ is the set of all geodesic rays in $\Gamma^{(\alpha)}$ starting from the root $(\varnothing,0)$, and it can be canonically embedded into $\N^{\N}$. 
Due to the transience of Brownian motion on $\Gamma^{(\alpha)}$, there is an a.s.~unique geodesic ray denoted by $W_\infty$ that is visited by $(W(t))_{t\geq 0}$ at arbitrarily large times. 
The distribution of the exit ray $W_\infty$ yields a (random) probability measure $\nu_{\alpha}$ on $\N^{\N}$. Thanks to \eqref{BM-CTGW}, we have in fact $\nu_{\alpha}=\mu_{\alpha}$, provided we think of both $\mu_{\alpha}$ and $\nu_{\alpha}$ as probability measures on $\N^{\N}$. 

\smallskip
\noindent{\bf Infinite continuous trees.} 
We let $\mathscr{T}$ be the set of all pairs $(\Pi,(z_v)_{v\in\Pi})$ that satisfy the following conditions:
\begin{enumerate}
\item[\rm(i)] $\Pi$ is an infinite discrete tree without leaves in the sense of Section~\ref{sec:treedelta};
\item[\rm(ii)] $z_{v}\in [0,\infty)$ for all $v \in \Pi$\,;
\item[\rm(iii)] $z_{\bar v}< z_v$ for every $v\in\Pi\backslash\{\varnothing\}$\,;
\item[\rm(iv)] for every $\mathbf{v}\in \Pi_{\infty}\colonequals\{(v_1,v_2,\ldots)\in \N^{\N}\colon (v_{1},v_{2},\ldots,v_{n})\in \Pi, \forall n\geq 1\}$, we have 
$$\lim_{n\to\infty} z_{(v_1,\ldots,v_n)} =\infty.$$
\end{enumerate}

If $(\Pi,(z_v)_{v\in\Pi})\in\mathscr{T}$, we consider the associated ``tree''
$$\t \colonequals \big(\{\varnothing\}\times [0,z_\varnothing]\big) \cup  \bigg(\bigcup_{v\in\Pi\backslash\{\varnothing\}} \{v\} \times (z_{\bar v}, z_v]\bigg),$$
equipped with the natural distance defined as above. 
The set $\Pi_{\infty}$ can be identified with the boundary $\partial \t$ of the tree $\t$, which is defined as the collection of all geodesic rays in $\t$. 
We keep the notation $H(x)=r$ for the height of a point $x=(v,r)\in\t$. 
The genealogical order on $\t$ is defined as previously and again is denoted by $\prec$. If $\mathbf{v}=(v_1,v_2,\ldots)\in \Pi_{\infty}$ and $x=(v,r)\in \t$, we write $x\prec \mathbf{v}$ if $v=(v_1,v_2,\ldots,v_k)$ for some integer $k\geq 0$.

When we say that we consider a tree $\t\in \mathscr{T}$, it means that we are given a pair $(\Pi,(z_v)_{v\in\Pi})$ satisfying the above properties, with $\t$ being the associated tree. 
Clearly, for every $\alpha \in (1,2]$, the CTGW tree $\Gamma^{(\alpha)}$ is a random element in $\mathscr{T}$, and we write $\Theta_{\alpha}(\mathrm{d}\t)$ for its distribution. 

Let us fix $\t=(\Pi,(z_v)_{v\in\Pi})\in \mathscr{T}$. Under our previous notation, the root $\varnothing$ has $k_{\varnothing}$ offspring. 
We denote by $\t_{(1)},\t_{(2)},\ldots,\t_{(k_{\varnothing})}$ the subtrees of $\t$ rooted at the first branching point $(\varnothing, z_\varnothing)$. 
To be more precise, for every $1\leq i\leq k_{\varnothing}$, we define the shifted discrete tree $\Pi[i]=\{v\in \mathcal{V}\colon iv\in \Pi\}$, and $\t_{(i)}$ is the infinite continuous tree corresponding to the pair $(\Pi[i],(z_{iv}-z_{\varnothing})_{v\in \Pi[i]})$.
Under $\Theta_{\alpha}(\mathrm{d}\t)$, the offspring number $k_{\varnothing}$ is distributed according to $\theta_{\alpha}$. The branching property of the CTGW tree means that under the conditional law $\Theta_{\alpha}(\mathrm{d}\t\!\mid\! k_{\varnothing})$, the subtrees $\t_{(1)},\ldots,\t_{(k_{\varnothing})}$ are i.i.d.~following the same law $\Theta_{\alpha}$.

If $r>0$, the level set of $\t$ at height $r$ is $\t_r= \{x\in\t \colon H(x)=r\}$. If $x\in\t_r$, let $\t[x]$ denote the subtree of descendants of $x$ in $\t$. To define it formally, we write $v_x$ for the unique element of $\v$ such that $x=(v_x,r)$, then $\t[x]$ is the tree associated with the pair $(\Pi[v_{x}], (z_{v_xv}-r)_{v\in\Pi[v_{x}]})$.
As usual, $\t_r[x]$ stands for the level set at height $r$ of the tree $\t[x]$. 

As we have seen in the Introduction, the martingale limit 
$$\mathcal{W}\a(\t)=\lim_{r\to \infty} e^{-\frac{r}{\al-1}}\# \t_r $$ 
exists $\Theta_\al(\mathrm{d}\t)$-a.s., and $\int \mathcal{W}\a(\t) \Theta_\al(\mathrm{d}\t)=1$. For every $x\in \t$, we similarly set 
$$\mathcal{W}\a(\t[x])=\lim_{r\to \infty} e^{-\frac{r}{\al-1}}\# \t_r[x].$$
If $\mathbf{v}\in \partial \t$ is a geodesic ray passing through $x$, let $\mathcal{B}(\mathbf{v},H(x))$ denote the set of geodesic rays in~$\t$ that coincide with $\mathbf{v}$ up to height $H(x)$. 
Then $\Theta_\al(\mathrm{d}\t)$-a.s., the uniform measure $\bar \omega\a_\t$ on $\partial \t$ is the probability measure on $\partial \t$ characterized by 
\begin{displaymath}
\bar \omega\a_\t(\mathcal{B}(\mathbf{v},H(x)))= \exp\Big(-\frac{H(x)}{\al-1}\Big)\frac{\mathcal{W}\a(\t[x])}{\mathcal{W}\a(\t)}
\end{displaymath}
for every $x\in \t \hbox{ and } \mathbf{v}\in\partial \t$ such that $x\prec \mathbf{v}$.

For a fixed infinite continuous tree $\t$, we define the harmonic measure $\mu_\t$ on $\partial \t$ as the distribution of the exit ray chosen by Brownian motion on~$\t$ (with drift $1/2$ towards infinity).

\subsection{The invariant measure and the size-biased CTGW tree}
We write
\begin{displaymath}
\mathscr{T}^{*}\colonequals \mathscr{T}\times \N^{\N}
\end{displaymath}
for the set of all pairs consisting of a tree $\t\in \mathscr{T}$ and a distinguished geodesic ray $\mathbf{v}$. We define a transformation $S$ on $\mathscr{T}^{*}$ by shifting $(\t,\mathbf{v}=(v_1,v_2,\ldots))$ at the first branching point of $\t$, that is, $S(\t,\mathbf{v})= (\t_{(v_1)}, \wt{\mathbf{v}})$, where $\wt{\mathbf{v}}=(v_2,v_3,\ldots)$ and $\t_{(v_1)}$ is the tree rooted at the first branching point of $\t$ that is chosen by $\mathbf{v}$. 

The next result extends Proposition 5 in~\cite{LIN1} from the Yule tree $\Gamma^{(2)}$ to the CTGW tree $\Gamma\a$. 

\begin{proposition}
\label{prop:unif-meas-inv}
For all $\al\in(1,2]$, the probability measure $\mathcal{W}\a(\t)\Theta_\al(\mathrm{d}\t)\bar \omega\a_\t(\mathrm{d}\mathbf{v})$ on $\mathscr{T}^{*}$ is invariant under $S$.
\end{proposition}

\begin{proof}
Fix a bounded measurable function $F$ on $\mathscr{T}^{*}$. 
Under $\Theta_\al(\rd\t)$, the height $z_{\varnothing}$ of the first branching point is exponentially distributed with mean~1, while the offspring number $k_\varnothing$ is distributed according to $\theta_\al$. Recall that $\sum k\theta_\al(k)=\frac{\al}{\al-1}$ and 
\begin{equation}
\label{eq:decomp-w}
\cw\a(\t) = \sum_{i=1}^{k_\varnothing} e^{-\frac{z_\varnothing}{\al-1}} \cw\a(\t_{(i)}).
\end{equation}
Using these remarks and the branching property of the CTGW tree, we see that
\begin{eqnarray*}
&& \int F\circ S (\t,\mathbf{v})\mathcal{W}\a(\t)\Theta_\al(\mathrm{d}\t) \bar\omega_\t \a(\mathrm{d}\mathbf{v}) \\
&= & \sum_{k=2}^{\infty}\theta_\al(k)\sum_{i=1}^{k}\int F(\t_{(i)},\mathbf{u}) e^{-\frac{z_{\varnothing}}{\alpha-1}}\mathcal{W}\a(\t_{(i)}) \Theta_\al(\mathrm{d}\t\!\mid k_\varnothing=k)\,\bar \omega_{\t_{(i)}}\a(\mathrm{d}\mathbf{u}) \\
&=& \Big(\sum_{k=2}^{\infty} k \theta_\al(k)\Big)\times \Big(\int_{0}^{\infty} e^{-\frac{z}{\alpha-1}}e^{-z}\mathrm{d}z\Big) \times\int F(\t,\mathbf{u}) \mathcal{W}\a(\t)\,\Theta_\al(\mathrm{d}\t)\,\bar \omega_{\t}\a(\mathrm{d}\mathbf{u})\\
&=& \int F(\t,\mathbf{u})\mathcal{W}\a(\t) \Theta_\al(\mathrm{d}\t) \bar \omega_\t\a(\mathrm{d}\mathbf{u}),
\end{eqnarray*}
which shows the required invariance.
\end{proof}

To fully understand this invariant measure, we construct a size-biased version $\widehat \Gamma\a$ of the CTGW tree $\Gamma\a$, following the popularized idea of size-biasing a Galton--Watson tree (see e.g.~\cite{CRW91,LPP95b}). 
Let us start with the root $\varnothing$. 
It lives for a random exponential lifetime with parameter $\frac{\al}{\al-1}$, then it dies and simultaneously gives birth to a random number $N_1$ of children, with $N_1$ distributed as the size-biased offspring number $\widehat N_\al$. 
We pick one of these children uniformly at random, say $v_1\in \{1,2,\ldots, N_1\}$. We give independently the other children independent descendant trees distributed as $\Gamma\a$, whereas $v_1$ lives for another independent exponential lifetime with parameter $\frac{\al}{\al-1}$ and then reproduces a random number $N_2$ of children with $N_2$ an independent copy of $N_1$. 
Again, we choose one of the children of $v_1$ uniformly at random, call it $v_2 \in \{1,2,\ldots, N_2\}$, and give the others independent CTGW descendant trees distributed as $\Gamma\a$. 
Repeating this procedure independently for infinitely many times, we obtain a random infinite tree $\widehat \Gamma\a$ which will be called the size-biased CTGW tree with parameter $\al\in (1,2]$. 

A formal definition of $\widehat \Gamma\a$ with the distinguished geodesic ray $\widehat{\mathbf{v}}=(v_1,v_2,\ldots)$ as a random element of $\mathscr{T}^*$ can be given similarly as in Section 2.6 of \cite{LIN2}. 
Notice that the successive heights $(Z_{v_k})_{k\geq 1}$ of the vertices on the ``spine'' $\widehat{\mathbf{v}}$ are distributed as a homogeneous Poisson process on $(0,\infty)$ with intensity $\frac{\al}{\al-1}$. 

Keeping Proposition~\ref{prop:unif-meas-inv} in mind, we show by the next result why it is convenient to think of $\widehat \Gamma\a$ when studying a random geodesic ray in $\Gamma\a$ sampled according to $\bar \omega\a$. 

\begin{lemma}
\label{lem:CTGWbias}
The pair $\big(\widehat \Gamma\a, \widehat{\mathbf{v}}\big)\in \mathscr{T}^*$ follows the distribution $\mathcal{W}\a(\t)\Theta_\al(\mathrm{d}\t)\bar \omega\a_\t(\mathrm{d}\mathbf{v})$.
\end{lemma}

\subsection{The size-biased reduced tree}
\label{sec:size-biased-reduced}

Recall the bijection $\Psi \colon (v,r)\in \Delta\a_0 \mapsto (v,-\log (1-r))\in \Gamma\a$ introduced in Section~\ref{sec:ctgwtree}. 
If we apply its inverse $\Psi^{-1}(v,s)=(v,1-e^{-s})$ to the size-biased CTGW tree $\widehat \Gamma\a$ and then take the natural compactification, we obtain a random compact rooted tree $\widehat \Delta\a$ called the size-biased reduced tree of parameter $\al\in(1,2]$. 
A Poissonian description of $\widehat \Delta\a$ is given as follows.  

First, under the mapping $\Psi^{-1}$ the geodesic ray $\widehat{\mathbf{v}}=(v_1,v_2,\ldots)$ in $\widehat \Gamma\a$ corresponds to a distinguished point at height 1 in $\widehat \Delta\a$ that we will still denote by $\widehat{\mathbf{v}}$. 
Along the ancestral line of $\widehat{\mathbf{v}}$ in $\widehat \Delta\a$, we keep the same notation $(v_k)_{k\geq 1}$ for the branching points, with their respective heights 
\[
Y_{v_k}\colonequals 1-\exp(- Z_{v_k})\,, \quad k\geq 1\,,
\]
increasing to 1 as $k\to \infty$. 
It is immediate to see that $(Y_{v_k})_{k\geq 1}$ is distributed on $(0,1)$ as a Poisson process with intensity measure $\frac{\al}{\al-1}(1-x)^{-1}\mathbf{1}_{x\in(0,1)}\mathrm{d}x$. 
Moreover, if we set $Y_{v_0}=0$, the sequence 
\[ 
V_k \colonequals \frac{Y_{v_k}-Y_{v_{k-1}}}{1-Y_{v_{k-1}}}, \qquad k\geq 1,
\]
consists of i.i.d.~random variables with the density function $\frac{\al}{\al-1}(1-x)^{\frac{1}{\al-1}}$ over $[0,1]$. 

For any $r>0$, we write $r\Delta_0\a$ for the ``same'' random tree as $\Delta_0\a$ with the distance $\mathbf{d}$ multiplied by the factor $r$, which means that the infinite discrete tree $\Pi\a$ associated to $\Delta_0\a$ remains the same while all the uniform random variables $U_v, v\in \mathcal{V}$ involved in the definition of $\Delta_0\a$ are replaced by $rU_v, v\in \mathcal{V}$. 

Conditionally on $(Y_{v_k})_{k\geq 1}$, independently to each branching point $v_k$ of height $Y_{v_k}$ on the ancestral line of $\widehat{\mathbf{v}}$ in $\widehat \Delta\a$, we graft a random number $J_k$ of independent trees, all distributed according to $(1-Y_{v_k})\Delta_0\a$. 
The random number $J_k$ is independent of those trees grafted at $v_k$, and it has the same distribution as $\widehat N_\al-1$.
As in the case of the size-biased CTGW tree $\widehat \Gamma\a$, we assume the independence of $J_k$ and the trees at grafted $v_k$ among all different levels $k \geq 1$. 

\begin{figure}[!h]
 \begin{center}
 \includegraphics[width=11cm]{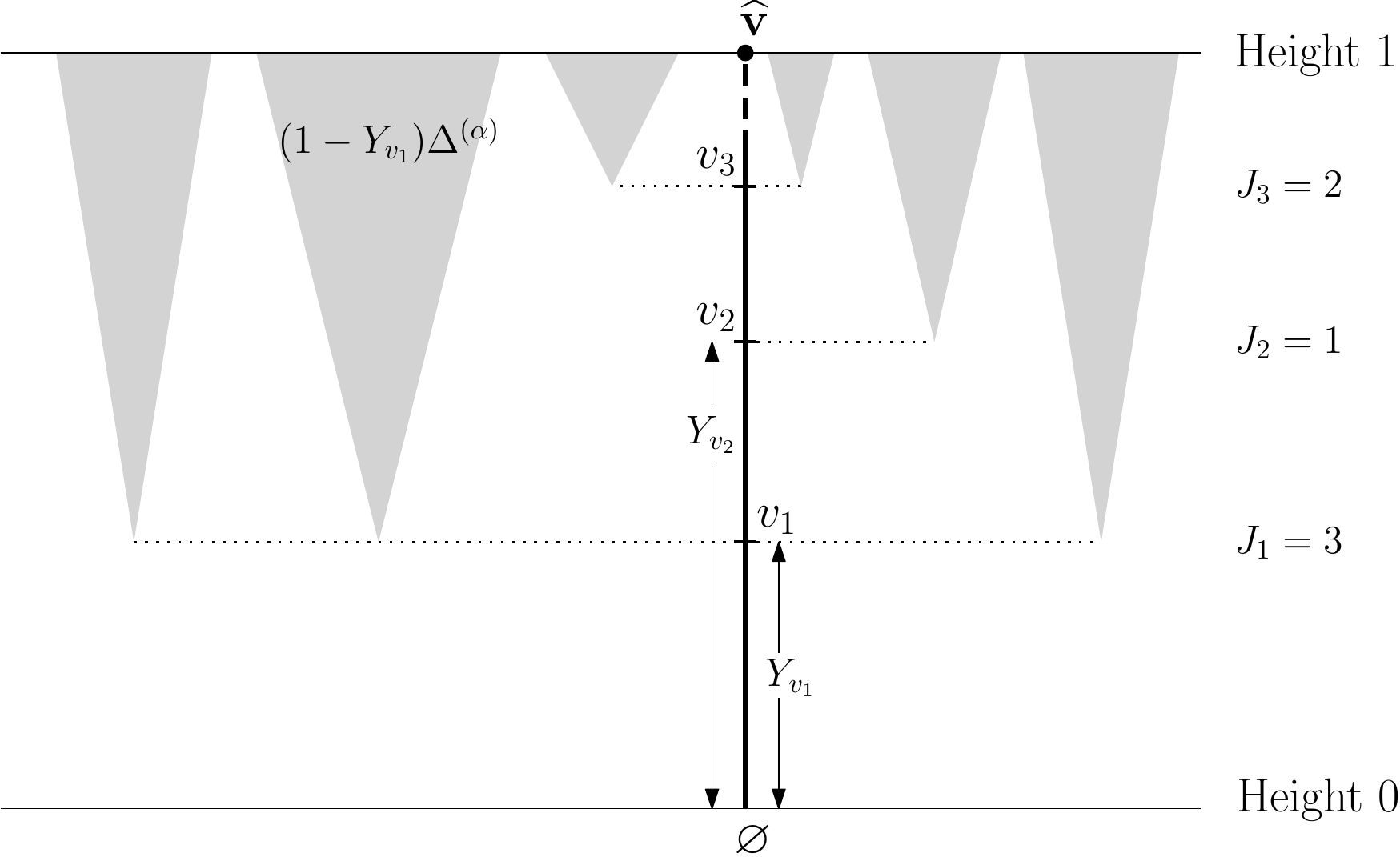}
 \caption{\label{fig:size-bias-Delta}Schematic representation of the size-biased reduced tree $\widehat \Delta\a$ when $1\leq \alpha<2$}
 \end{center}
 \end{figure}

To finally get the size-biased reduced tree $\widehat \Delta\a$, we take the completion of the tree obtained above with respect to its natural distance (see Fig.~\ref{fig:size-bias-Delta}). 
This compactification is equivalent to adding all the points on the boundary of the grafted trees. 
All these points together with the distinguished one $\widehat{\mathbf{v}}$ form the boundary $\partial \widehat \Delta\a$, which is defined as the set of all points in $\widehat \Delta\a$ that are at height 1. 

\rem The $\alpha$-stable continuum random tree arises as the scaling limit of rescaled discrete Galton--Watson trees with a critical offspring distribution in the domain of attraction of a stable distribution of index $\al\in(1,2]$ (see e.g.~the monograph~\cite{DLG02} of Duquesne and Le Gall).
Under the same assumption, if we condition the critical Galton--Watson tree to survive up to generation~$n$ and rescale it by the factor $n^{-1}$, after letting $n\to \infty$, we get as limit an $\al$-stable tree conditioned to reach height 1 (see Theorem 4.1 in~\cite{DLG05}), from which the reduced stable tree $\Delta\a$ can be obtained as a pruned tree by keeping only the points at height less than 1 that possess a descendant that reaches height 1. 
The probabilistic structure of the size-biased reduced tree $\widehat \Delta\a$ can also be derived from a spinal decomposition (Theorem 4.5 in~\cite{DLG05}) of the conditional $\al$-stable tree along the ancestral line of a typical point at height 1. 
Since this point of view is not needed for proving our main results, we omit the details. 

\subsection{The continuous conductance}
\label{sec:conductance}

Recall that the random variable $\mathcal{C}^{(\alpha)}$ is defined, from an electric network point of view, as the effective conductance between the root and the boundary $\partial \Delta^{(\alpha)}$ in $\Delta\a$. 
This is equivalent to defining $\mathcal{C}^{(\alpha)}$ as the mass assigned by the excursion measure away from the root for Brownian motion (with no drift) on $\Delta\a$, to the set of trajectories that reach height~1 before coming back to the root. 
This probabilistic definition of the conductance has the advantage of being easily generalized to a more general setting. 
If $\t$ is an infinite continuous tree, we define its conductance $\mathcal{C}(\t)$ between the root and $\partial\t$ as the mass assigned by the excursion measure away from the root for Brownian motion with drift $1/2$ on $\t$, to the set of trajectories that go to infinity without returning to the root. 
By the correspondence between the CTGW tree $\Gamma\a$ and the reduced stable tree $\Delta\a$, we can verify that $\mathcal{C}(\Gamma\a)$ has the same law as $\mathcal{C}^{(\alpha)}$. 
See~\cite[Section 2.3]{CLG13} for details.

In Section 2.3 of \cite{LIN1}, it is shown that for $\alpha \in (1,2]$ the law of $ \mathcal{C}\a$ is characterized by the distributional identity~\eqref{eq:rde} in the class of all probability measures on $[1,\infty]$. 
Now let us take into account the martingale limit $\cw\a$. 
From~\eqref{eq:decomp-w} and the branching property of $\Gamma\a$, it follows that the joint distribution of $(\cw\a, \cc\a)$ satisfies the distributional equation
\begin{equation}
\label{eq:rde-joint}
\big(\mathcal{W}\a, \mathcal{C}\a \big)\,\overset{(\mathrm{d})}{=\joinrel=}\, \bigg ( (1-U)^{\frac{1}{\al-1}} (\mathcal{W}\a_{1}+\cdots+\mathcal{W}\a_{N_\al}), \Big(U + \frac{1-U}{\mathcal{C}\a_1+\cdots+ \mathcal{C}\a_{N_\al}}\Big)^{-1}\bigg),
\end{equation}
in which $(\mathcal{W}\a_i, \cc\a_i)_{i\geq 1}$ are i.i.d.~copies of $(\mathcal{W}\a, \cc\a)$, the integer-valued random variable $N_\al$ is distributed according to $\theta_\al$, and $U$ is uniformly distributed over $[0,1]$. All these random variables are supposed to be independent. 

Under $\P$, we define for every $\al \in (1,2]$ a positive random variable $\widehat{\cc}\a$ distributed as $\mathcal{C}(\t)$ under the probability measure $\mathcal{W}\a(\t)\Theta_\al(\mathrm{d}\t)$. 
If the size-biased reduced tree $\widehat \Delta\a$ is viewed as an electric network of ideal resistors with unit resistance per unit length, we will see that its effective conductance between the root and $\partial \widehat \Delta\a$ has the same law as $\widehat \cc\a$.  
For this reason, we sometimes call $\widehat \cc\a$ the size-biased conductance.

\begin{lemma}
\label{lem:conduc-bias}
The random variable $\widehat{\mathcal{C}}\a$ satisfies the distributional identity
\begin{equation}
\label{eq:c*-rde}
\widehat{\mathcal{C}}\a \,\overset{(\mathrm{d})}{=\joinrel=}\, \bigg(V_\al + \frac{1-V_\al}{\widehat{\mathcal{C}}\a+ \mathcal{C}\a_2+\cdots +\cc\a_{\widehat{N}_\al}}\bigg)^{-1},
\end{equation}
where in the right-hand side the random variable $V_\al$ has density function $\frac{\al}{\al-1}(1-x)^{\frac{1}{\al-1}}$ over $[0,1]$, the integer-valued random variable $\widehat{N}_\al$ has the size-biased distribution of $\theta_\al$, and $(\cc\a_i)_{i\geq 2}$ are i.i.d.~copies of $\cc\a$. 
All these random variables $V_\al, \widehat{N}_\al, \widehat{\mathcal{C}}\a$ and $(\cc\a_i)_{i\geq 2}$ are independent. 
\end{lemma}

\begin{proof}
The law of $\widehat{\mathcal{C}}\a$ is determined by $\E [g(\widehat{\mathcal{C}}\a)]= \E[\mathcal{W}\a\,g(\mathcal{C}\a)]$ for every nonnegative measurable function $g$ on $\R_+$.
Using (\ref{eq:rde-joint}) under the same notation, we have
\begin{eqnarray*}
\E\Big[g(\widehat{\mathcal{C}}\a)\Big] &=& \E\Bigg[(1-U)^{\frac{1}{\al-1}} (\mathcal{W}\a_{1}+\cdots+\mathcal{W}\a_{N_\al}) \,g\Bigg(\bigg(U + \frac{1-U}{\mathcal{C}\a_1+\cdots+ \mathcal{C}\a_{N_\al}}\bigg)^{-1}\Bigg)\Bigg]\\
&=& \E \Bigg[N_\al (1-U)^{\frac{1}{\al-1}} \mathcal{W}\a_{1} \,g\Bigg(\bigg(U + \frac{1-U}{\mathcal{C}\a_1+\cdots+ \mathcal{C}\a_{N_\al}}\bigg)^{-1}\Bigg)\Bigg].
\end{eqnarray*}
Defining $\widehat{N}_\al$ as the size-biased version of $N_\al$, we see that 
\begin{displaymath}
\E\Big[g(\widehat{\mathcal{C}}\a)\Big] = \E [ N_\al] \cdot \E \Bigg[ (1-U)^{\frac{1}{\al-1}} \mathcal{W}\a_{1} \,g\Bigg(\bigg(U + \frac{1-U}{\mathcal{C}\a_1+\cdots+ \mathcal{C}\a_{\widehat{N}_\al}}\bigg)^{-1}\Bigg)\Bigg].
\end{displaymath}
Recall that $\E[N_\al]=\frac{\al}{\al-1}=\E[(1-U)^\frac{1}{\al-1}]^{-1}$. 
The statement of the lemma thus follows if we let $V_{\alpha}$ be a random variable independent of $\widehat{N}_\al, \widehat{\mathcal{C}}\a$ and $(\cc\a_i)_{i\geq 2}$, satisfying
\begin{equation}
\label{eq:V-alpha}
\E[f(V_{\alpha})]= \E\bigg[\frac{\alpha}{\alpha-1}(1-U)^{\frac{1}{\alpha-1}}f(U)\bigg]
\end{equation}
for every nonnegative measurable function $f$.
\end{proof}

The law $\gamma_\al$ of the conductance $\mathcal{C}\a$ has been studied at length in Proposition 2.1 of \cite{LIN1}. 
We now discuss some similar properties of~$\widehat{\mathcal{C}}\a$. 
For every $v \in (0,1)$, $n \geq 2$, $x \in [1,\infty)$ and $(c_{i})_{i\geq 2}\in [1,\infty)^{\N}$, we define 
\begin{equation} 
\label{eq:G-defi}
G(v,n,x, (c_{i})_{i\geq 2}) \colonequals \left( v + \frac{1-v}{x+c_{2}+\cdots+c_{n}}\right)^{-1}, 
\end{equation} 
so that \eqref{eq:c*-rde} can be reformulated as 
\begin{equation}
\label{rde-bis}
\wcc \a \overset{(\mathrm{d})}{=\joinrel=} G(V_\al, \widehat{N}_{\alpha}, \wcc\a, (\mathcal{C}\a_{i})_{i\geq 2})
\end{equation} 
where $V_\al, \widehat{N}_{\alpha}, \wcc\a, (\mathcal{C}\a_{i})_{i\geq 2}$ are as in~\eqref{eq:c*-rde}. Let $ \mathscr{M}$ be the set of all probability measures on $[1, \infty]$ and let $\widehat{\Phi}_{\alpha} \colon \mathscr{M} \to \mathscr{M}$ map a probability distribution $\sigma$ to 
\begin{equation*} 
\widehat{\Phi}_{\alpha}(\sigma) = \mathsf{Law} \big(G(V_\al,\widehat{N}_{\alpha},X, (\cc_{i}\a)_{i\geq 2})\big) 
\end{equation*} 
where $X$ is distributed according to $\sigma$ and independent of $V_\al, \widehat{N}_{\alpha}, (\mathcal{C}\a_{i})_{i\geq 2}$.

\begin{proposition}
\label{prop:c*-law}
We fix the parameter $\alpha \in (1,2]$.
\begin{enumerate}
\item[(1)] The distributional equation~(\ref{eq:c*-rde}) characterizes the law $\widehat \gamma_\al$ of $\widehat{\mathcal{C}}\a$ in the sense that, $\widehat \gamma_\al$ is the unique fixed point of the mapping $\widehat\Phi_\al$ on $\mathscr{M}$, and for every $\sigma \in \mathscr{M}$, the $k$-th iterate $\widehat \Phi_\al^k(\sigma)$ converges to $\widehat \gamma_\al$ weakly as $k \to \infty$.
\item[(2)] The law $\widehat \gamma_\al$ has a continuous density over $[1,\infty)$. 
\item[(3)] For any monotone continuously differentiable function $g\colon [1,\infty)\to \R_{+}$, we have
\begin{equation}
\label{eq:g-c-c*}
\E\Big[\widehat{\mathcal{C}}\a(\widehat{\mathcal{C}}\a-1)g'(\widehat{\mathcal{C}}\a)\Big]+\frac{\al}{\al-1} \,\E\Big[g(\widehat{\mathcal{C}}\a)\Big]=\,\frac{\al}{\al-1} \E\Big[g(\widehat{\mathcal{C}}\a+\mathcal{C}\a_2+\cdots+\cc\a_{\widehat N_\al})\Big],
\end{equation}
where $\widehat{\mathcal{C}}\a$, $\widehat{N}_\al$, and $(\mathcal{C}\a_i)_{i\geq 2}$ are as in \eqref{rde-bis}.
\item[(4)] We define, for all $\ell\geq 0$, the Laplace transforms $\varphi_\al(\ell)\colonequals \E[\exp(-\ell\, \mathcal{C}\a/2)]$ and
\begin{displaymath}
\widehat \varphi_\al(\ell)=\E[\exp(-\ell\, \widehat{\mathcal{C}}\a/2)]\colonequals \int_1^{\infty} e^{-\ell x/2}\,\widehat \gamma_\al(\mathrm{d}x).
\end{displaymath}
Then $\widehat \varphi_\al$ solves the linear differential equation
\begin{equation}
\label{eq:laplace}
2\ell\, \psi''(\ell)+\ell\,\psi'(\ell)-\frac{\al}{\al-1} (1-\varphi_\al(\ell))^{\al-1}\psi(\ell)=0.
\end{equation}
\item[(5)] If $1<\al<2$, only the first moment of $\wcc\a$ is finite, whereas all moments of $\wcc^{(2)}$ are finite. 
\end{enumerate}
\end{proposition}

\proof The results for $\al=2$ are stated without proof in~\cite{LIN2}. 
The arguments given below is similar in spirit to that of Proposition 2.1 in~\cite{LIN1}, but due to the fact that $\widehat N_\al$ has an infinite mean for $\al<2$, sharper estimates are necessary in the present setting. 

(1) First of all, the stochastic partial order $\preceq$ on $ \mathscr{M}$ is defined by saying that $ \sigma \preceq \sigma'$ if and only if there exists a coupling $(X,Y)$ of $\sigma$ and $\sigma'$ such that a.s.~$X \leq Y$. It is clear that for any $\alpha\in [1,2]$, the mapping $\widehat{\Phi}_{\alpha}$ is increasing for the stochastic partial order. 
We endow the set $\mathscr{M}_{1}$ of all probability measures on $[1, \infty]$ that have a finite first moment with the $1$-Wasserstein metric
\[
\mathrm{d}_{1}( \sigma,\sigma') \colonequals  \inf\big \{ E\big[|X-Y|\big] \colon  (X,Y)\mbox{ coupling of }(\sigma,\sigma')\big\}.  
\]
The metric space $(\mathscr{M}_{1}, \mathrm{d}_{1})$ is Polish and its topology is finer than the weak topology on $ \mathscr{M}_{1}$. 

Let us show that $\widehat{\Phi}_{\alpha}$ maps $\mathscr{M}_{1}$ into $ \mathscr{M}_{1}$ for any $\alpha\in(1,2]$. For any $a,b>0$ and $0<r<1$, the weighted harmonic-geometric means inequality says that
$$ \Big(\frac{r}{a}+\frac{1-r}{b}\Big)^{-1}\leq a^r b^{1-r}.$$
We fix some $r\in (2-\al,1)$ and apply the preceding inequality to obtain
$$ \bigg(V_\al + \frac{1-V_\al}{X+\mathcal{C}\a_2+\cdots +\cc\a_{\widehat{N}_\al}}\bigg)^{-1} \leq  \Big(\frac{r}{V_\al}\Big)^r \Big(\frac{1-r}{1-V_\al}\Big)^{1-r} \big(X+\mathcal{C}\a_2+\cdots +\cc\a_{\widehat{N}_\al}\big)^{1-r}.$$
It suffices to show that the right-hand side of the last display has a finite first moment provided that $X$ has a finite mean. Notice that by \eqref{eq:V-alpha}, 
$$ \E\bigg[ \frac{1}{(V_\al)^r (1-V_\al)^{1-r}}\bigg]= \E\bigg[ \frac{\al}{\al-1} (1-U)^{\frac{1}{\al-1}} \frac{1}{U^r (1-U)^{1-r}}\bigg] $$
is finite. 
On the other hand, since a.s.~$X\geq 1$ and $\mathcal{C}\a_i \geq 1$ for all $i\geq 2$, 
\begin{eqnarray}
\E\Big[ \big(X+\mathcal{C}\a_2+\cdots +\cc\a_{\widehat{N}_\al}\big)^{1-r} \Big] & \leq& \E\Bigg[ \frac{X+\mathcal{C}\a_2+\cdots +\cc\a_{\widehat{N}_\al}}{\big(\widehat{N}_\al\big)^r}\Bigg] \nonumber \\
&=& \frac{\al-1}{\al} \sum_{k\geq 2} \frac{k \theta_\al(k)}{k^r} \Big(\E[X]+\E\big[\mathcal{C}\a_2\big]+\cdots +\E\big[\cc\a_k\big] \Big). \label{eq:finite-r}
\end{eqnarray}
It is shown in \cite[Proposition 2.1]{LIN1} that $\cc\a$ has a finite mean. 
Meanwhile, we know by Stirling's formula that $\theta_{\alpha}(k)=O(k^{-(1+\alpha)})$ when $k\to \infty$. As $r>2-\al$, it follows that $\sum k^{2-r}\theta_\al(k)<\infty$. 
The sum in \eqref{eq:finite-r} is finite and therefore $G(V_\al,\widehat{N}_{\alpha},X, (\cc_{i}\a)_{i\geq2})$ has a finite first moment.

Next, we observe that $\widehat\Phi_{\alpha}$ is strictly contractant on $(\mathscr{M}_{1}, \mathrm{d}_{1})$. 
To see this, let $(X,Y)$ be a coupling between $\sigma,\sigma' \in \mathscr{M}_{1}$ under the probability measure $\mathbb{P}$, and assume that $V_\al, \widehat N_{\alpha}, (\cc\a_i)_{i\geq 2}$ and $(X,Y)$ are independent under $\mathbb{P}$. 
Then $G(V_\al,\widehat N_{\alpha}, X, (\cc\a_i)_{i\geq 2})$ and $G(V_\al,\widehat N_{\alpha}, Y, (\cc\a_i)_{i\geq 2})$ provide a coupling of $ \widehat\Phi_{\alpha}(\sigma)$ and $\widehat \Phi_{\alpha}(\sigma')$.  
Using the fact that a.s.~$X, Y, \cc\a_i \geq 1$, we have
\begin{eqnarray*} 
 && \left| G(V_\al,\widehat N_{\alpha}, X, (\cc\a_i)_{i\geq 2}) - G(V_\al,\widehat N_{\alpha}, Y, (\cc\a_i)_{i\geq 2}) \right|\\
 &=& \Bigg| \bigg(V_\al+ \frac{1-V_\al}{X+\cc\a_2+\cdots+\cc\a_{\widehat N_{\alpha}}}\bigg)^{-1} -  \bigg(V_\al+ \frac{1-V_\al}{Y +\cc\a_2+\cdots+\cc\a_{\widehat N_{\alpha}}}\bigg)^{-1} \Bigg|\\
 &=& \Bigg|\frac {(1-V_\al)(X-Y)}{\big(V_\al\big(X+\cc\a_2+\cdots+\cc\a_{\widehat N_{\alpha}}\big)+1-V_\al\big)\big(V_\al\big(Y+\cc\a_2+\cdots+\cc\a_{\widehat N_{\alpha}}\big)+1-V_\al\big)}\Bigg|\\
 &\leq &  \frac{1-V_\al}{(1+(\widehat N_{\alpha}-1)V_\al)^2} |X-Y|.
\end{eqnarray*}
By definition, 
\begin{eqnarray*} 
\E\bigg[\frac{1-V_\al}{(1+(\widehat N_{\alpha}-1)V_\al)^2}\bigg] & = & \frac{\al-1}{\al}\sum_{k\geq 2} k\theta_\al(k)\, \E\bigg[\frac{1-V_\al}{(1+(k-1)V_\al)^2}\bigg] \\
& =& \sum_{k \geq 2} \theta_\al(k) \, \E\bigg[\frac{k(1-U)^{\frac{\al}{\al-1}}}{(1+(k-1)U)^2}\bigg] \\
& \leq & \sum_{k \geq 2} \theta_\al(k) \, \E\bigg[\frac{k(1-U)}{(1+(k-1)U)^2}\bigg]
\end{eqnarray*}
where $U$ is uniformly distributed over $[0,1]$. 
Notice that for for any integer $k\geq 2$,
\begin{displaymath}
\mathbb{E}\bigg[\frac{k(1-U)}{(1+(k-1)U)^{2}}\bigg]=1+\frac{k-1-k\log k}{(k-1)^{2}}<1.
\end{displaymath}
Thus, taking expectation and minimizing over the couplings between $\sigma$ and $\sigma'$, we get 
\begin{eqnarray*} 
\mathrm{d}_{1}( \widehat\Phi_{\alpha}(\sigma), \widehat\Phi_{\alpha}(\sigma'))
&\leq& \mathbb{E}\bigg[\frac{(1-V_\al)}{(1+(\widehat N_{\alpha}-1)V_\al)^{2}}\bigg]\mathrm{d}_{1}(\sigma,\sigma')\\
&\leq & \bigg(1+\mathbb{E}\Big[\frac{N_{\alpha}-1-N_{\alpha}\log N_{\alpha}}{(N_{\alpha}-1)^{2}}\Big]\bigg)\mathrm{d}_{1}(\sigma,\sigma') \;=:\; c_{\alpha}\, \mathrm{d}_{1}(\sigma,\sigma')
\end{eqnarray*}
with $c_{\alpha}<1$. So for $\alpha \in (1,2]$, the mapping $\widehat \Phi_{\alpha}$ is contractant on $ \mathscr{M}_{1}$ and by completeness it has a unique fixed point $\widetilde\gamma_{\alpha}$ in $ \mathscr{M}_{1}$. Furthermore, for every $ \sigma \in \mathscr{M}_{1}$, we have $ \widehat\Phi_{\alpha}^k(\sigma) \to \widetilde\gamma_{\alpha}$ for the metric $ \mathrm{d}_{1}$, hence also weakly, as $k \to \infty$.
However, $\widehat\gamma_{\alpha}$ is also a fixed point of $\widehat \Phi_{\alpha}$ according to \eqref{rde-bis}.  
The equality $\widehat\gamma_{\alpha}=\widetilde\gamma_{\alpha}$ will follow if we can verify that $\widetilde\gamma_{\alpha}$ is the unique fixed point of $\widehat\Phi_{\alpha}$ in~$\mathscr{M}$. 
To this end, it will be enough to show that $ \widehat\Phi_{\alpha}^k(\sigma) \to \widetilde\gamma_{\alpha}$ as $k\to\infty$, for every $\sigma\in \mathscr{M}$. 

Let us apply $\widehat \Phi_{\alpha}$ to the Dirac measure $\delta_{\infty}$ at infinity to see that
\begin{eqnarray*}
\widehat \Phi_{\alpha}(\delta_{\infty}) &= &\mathsf{Law} \big(\widetilde V_\al^{-1}\big)\,, \\
\widehat \Phi_{\alpha}^{2}(\delta_{\infty}) &= &\mathsf{Law} \Bigg(\bigg(V_\al+\frac{1-V_\al}{\widetilde V_\al^{-1}+\cc\a_2+\cdots+\cc\a_{\widehat N_{\alpha}}}\bigg)^{-1}\Bigg)\,,
\end{eqnarray*}
where $\widetilde V_\al$ is a copy of $V_\al$ being independent of $V_\al, \widehat N_{\alpha}$ and $(\cc\a_i)_{i\geq 2}$ under $\mathbb{P}$. 
We claim that $\widehat \Phi_{\alpha}^{2}(\delta_{\infty})$ has a finite first moment. In fact, applying \eqref{eq:V-alpha} yields
\begin{align*}
\E \Bigg[ \bigg(V_\al+\frac{1-V_\al}{\widetilde V_\al^{-1}+\cc\a_2+\cdots+\cc\a_{\widehat N_{\alpha}}}\bigg)^{-1}\Bigg] &=\, \E \Bigg[ \frac{\al}{\al-1} (1-U)^\frac{1}{\al-1} \bigg(U+\frac{1-U}{\widetilde V_\al^{-1}+\cc\a_2+\cdots+\cc\a_{\widehat N_{\alpha}}}\bigg)^{-1}\Bigg] \\
&\leq\,  \E \Bigg[ \frac{\al}{\al-1} \bigg(U+\frac{1-U}{\widetilde V_\al^{-1}+\cc\a_2+\cdots+\cc\a_{\widehat N_{\alpha}}}\bigg)^{-1}\Bigg] \\
&= \, \sum_{k\geq 2} k\theta_\al(k)\E\left[ \bigg(U+\frac{1-U}{\widetilde V_\al^{-1}+\cc\a_2+\cdots+\cc\a_k} \bigg)^{-1}\right]. 
\end{align*}
We integrate with respect to $U$ to obtain 
\begin{align*}
\E\left[ \bigg(U+\frac{1-U}{\widetilde V_\al^{-1}+\cc\a_2+\cdots+\cc\a_k} \bigg)^{-1}\right] & = \,\E\Bigg[\frac{\log \big( \widetilde V_\al^{-1}+\cc\a_2+\cdots+\cc\a_k\big)}{1-(\widetilde V_\al^{-1}+\cc\a_2+\cdots+\cc\a_k)^{-1}} \Bigg] \\
&\leq\, 2\, \E\left[\log \big( \widetilde V_\al^{-1}+\cc\a_2+\cdots+\cc\a_k\big)\right].
\end{align*}
Using the inequality $\log(a+b)\leq \sqrt a+\log b$ for $a,b\geq 1$, we see that 
$$\E\left[\log \big( \widetilde V_\al^{-1}+\cc\a_2+\cdots+\cc\a_k\big) \right] \leq \E\left[\widetilde V_\al^{-1/2}\right] + \E \left[\log \big(\cc\a_2+\cdots+\cc\a_k \big)\right].$$
The random variable $\widetilde V_\al ^{-1/2}$ has clearly a finite mean. 
Meanwhile, for any $r\in(0,1)$, we have $\log a \leq \frac{1}{r} a^r$ for every $a\geq 1$, and thus
$$\E \left[\log \big(\cc\a_2+\cdots+\cc\a_k\big)\right]\leq \frac{1}{r}\E\left[\big(\cc\a_2+\cdots+\cc\a_k\big)^r\right]\leq \frac{1}{r}\E\left[\cc\a_2+\cdots+\cc\a_k\right]^r.$$
If we take $r<\al-1$, by the same arguments following \eqref{eq:finite-r} we deduce that 
\begin{equation}
\label{eq:logsum-finite}
\sum_{k\geq 2} k\theta_\al(k)\E \left[\log \big(\cc\a_2+\cdots+\cc\a_k\big)\right]\leq \frac{1}{r} \sum_{k\geq 2} k\theta_\al(k) \E\left[\cc\a_2+\cdots+\cc\a_k\right]^r <\infty.
\end{equation}
Therefore, we get $\widehat\Phi_{\alpha}^2(\delta_\infty)\in\mathscr{M}_1$. 
By monotonicity, $\widehat\Phi_{\alpha}^2(\sigma)\in\mathscr{M}_1$ for every $\sigma\in\mathscr{M}$, and from the preceding results we have $\widehat\Phi_{\alpha}^k(\sigma) \to \widetilde\gamma_{\alpha}$ for every $\sigma\in \mathscr{M}$. This implies that $\widehat\gamma_{\alpha}=\widetilde\gamma_{\alpha}$ is the unique fixed point of $\widehat\Phi_{\alpha}$ in $\mathscr{M}$.

(2) For every $t\in [1,\infty)$ we set $\widehat F_{\alpha}(t) \colonequals \P(\widehat\cc\a\geq t)$, and 
\[
\widehat F^{(k)}_{\alpha}(t) \colonequals \mathbb{P}( \widehat \cc\a_{1}+\cc\a_{2}+\cdots+\cc\a_{k} \geq t)
\]
for every integer $k\geq 2$.
It follows from (\ref{eq:c*-rde}) that, for every $t>1$,
\begin{eqnarray}
\widehat F_{\alpha}(t)  &=& \mathbb{P}\bigg( V_{\alpha} + \frac{1-V_{\alpha}}{ \widehat \cc^{(\alpha)} +\cc\a_{2}+\cdots+\cc\a_{\widehat N_\al}} \leq \frac{1}{t} \bigg) \nonumber \\
&=&  \sum\limits_{k=2}^{\infty} k \theta_{\alpha}(k) \int_{0}^{1/t}\mathrm{d}v \,(1-v)^{\frac{1}{\alpha-1}} \,\widehat F_{\alpha}^{(k)}\Big(\frac{t-vt}{1-vt}\Big) \nonumber \\ 
&=& \Big(\frac{t-1}{t}\Big)^{\frac{\alpha}{\alpha-1}}\int_{t}^{\infty} \mathrm{d}x\, \frac{x^{\frac{1}{\alpha-1}}}{(x-1)^{\frac{2\alpha-1}{\alpha-1}}} \,\E\big[N_{\alpha}\widehat F_{\alpha}^{(N_{\alpha})}(x)\big], \label{eq:repart}
\end{eqnarray}
where $N_\al$ is distributed according to $\theta_\al$. 
Since $\widehat F^{(k)}_{\alpha}(t)=1$ for every $t\in[1,2]$ and $k\geq 2$, we obtain from the last display that 
\begin{equation}
\label{eq:valueon[1,2]}
\widehat F_{\alpha}(t) = 1-\Big(\frac{t-1}{t}\Big)^{\frac{\al}{\al-1}} A_0\a,\qquad \forall t\in[1,2],
\end{equation}
where 
$$A_0\a \colonequals 2^{\frac{\al}{\al-1}}-\int_2^{\infty} \mathrm{d}x\, \frac{x^{\frac{1}{\alpha-1}}}{(x-1)^{\frac{2\alpha-1}{\alpha-1}}} \,\E\big[N_{\alpha}\widehat F_{\alpha}^{(N_{\alpha})}(x)\big] \in [1,2^{\frac{\al}{\al-1}}).$$
From \eqref{eq:repart} and \eqref{eq:valueon[1,2]}, we observe that $\widehat F_{\alpha}$ is continuous on $[1,\infty)$. 
Recall that $\cc\a$ has a continuous density (see~\cite[Proposition 2.1]{LIN1}). 
It follows that all functions $\widehat F^{(k)}_{\alpha}, k\geq 2$ are continuous on $[1,\infty)$. 
By dominated convergence the function 
$$x\mapsto \E\big[N_\al \widehat F_{\alpha}^{(N_{\alpha})}(x)\big]$$ 
is also continuous on $[1,\infty)$. Using \eqref{eq:repart} and \eqref{eq:valueon[1,2]} again, we see that $\widehat F_{\alpha}$ is continuously differentiable on $[1,\infty)$. Consequently $\widehat \gamma_{\alpha}$ has a continuous density on $[1,\infty)$. 

(3) To derive the identity \eqref{eq:g-c-c*}, we first differentiate (\ref{eq:repart}) with respect to $t$ to see that 
\begin{displaymath}
t(t-1)\frac{\mathrm{d}\widehat F_{\alpha}(t)}{\mathrm{d}t}-\frac{\alpha}{\alpha-1}\widehat F_{\alpha}(t)=-\E\big[N_{\alpha}\widehat F_{\alpha}^{(N_{\alpha})}(t)\big]
\end{displaymath}
holds for $t\in(1,\infty)$. Let $g\colon [1,\infty) \to \mathbb{R}_{+}$ be a monotone continuously differentiable function. 
We multiply both sides of the last display by $g'(t)$ and integrate for $t$ from $1$ to $\infty$ to obtain
\begin{equation*} 
\mathbb{E}\big[ \widehat{\mathcal{C}}^{(\alpha)}( \widehat{\mathcal{C}}^{(\alpha)}-1) g'(\widehat{\mathcal{C}}^{(\alpha)})\big] + \frac{\alpha}{\alpha-1}\mathbb{E} \big[ g(\widehat{\mathcal{C}}^{(\alpha)})\big] =  \E\big[N_{\alpha} \,g(\widehat{\mathcal{C}}^{(\alpha)}+ \mathcal{C}\a_{2}+\cdots+\mathcal{C}\a_{N_{\alpha}})\big],
\end{equation*} 
which gives \eqref{eq:g-c-c*}. 

(4) Taking $g(x)=\exp(- x\ell/2 )$ for $\ell >0$ in~(\ref{eq:g-c-c*}), we obtain \eqref{eq:laplace} by using the generating function of $N_{\alpha}$ given in (\ref{eq:gene-fct}). 

(5) The first moment of $\widehat{\mathcal{C}}\a$ is finite because $\widehat\gamma_{\alpha}\in\mathscr{M}_{1}$. Taking $g(x)=x$ in~(\ref{eq:g-c-c*}) yields
\begin{displaymath}
\E\big[(\widehat{\mathcal{C}}\a)^{2}\big]=\E\big[\widehat \cc\a\big]+\frac{\al}{\al-1}\E\big[\widehat N_\al-1\big] \E\big[\cc\a\big]\,.
\end{displaymath}
For every $\alpha\in(1,2)$, as $\E\big[\widehat N_{\alpha}\big]=\infty$ the second moment of $\widehat{\mathcal{C}}\a$ is infinite.
Finally, by Proposition~6 in~\cite{CLG13}, all moments of $\cc^{(2)}$ are finite. Using
$$\E\big[\widehat{\mathcal{C}}^{(2)}(\widehat{\mathcal{C}}^{(2)}-1)g'(\widehat{\mathcal{C}}^{(2)})\big]+2\,\E\big[g(\widehat{\mathcal{C}}^{(2)})\big]= 2\, \E\big[g(\widehat{\mathcal{C}}^{(2)}+\cc^{(2)})\big]$$
with $g(x)=x^n, n\geq 1$, one can easily show by induction that all moments of $\widehat{\mathcal{C}}^{(2)}$ are finite.
\endproof
 

\subsection{Proof of Theorem~\ref{thm:dim-continu}}
\label{sec:proof-continu}
The following proof of Theorem~\ref{thm:dim-continu} for all $\al\in(1,2]$ proceeds in much the same way as in the case $\al=2$. 
We skip some details, for which we refer the reader to Section 2.5 in~\cite{LIN2}.
  
First, we have seen in Proposition~\ref{prop:unif-meas-inv} that the probability measure $\mathcal{W}\a(\t)\Theta_\al(\mathrm{d}\t)\bar \omega\a_\t(\mathrm{d}\mathbf{v})$ on $\mathscr{T}^{*}$ is invariant under the shift $S$. Taking into account that $\mathcal{W}\a(\t)>0$, $\Theta_\al(\mathrm{d}\t)$-a.s., we can verify, in a similar fashion as in \cite[Proposition 2.6]{LIN1}, that the shift $S$ acting on the probability space $(\mathscr{T}^{*},\mathcal{W}\a(\t)\Theta_\al(\mathrm{d}\t)\bar \omega\a_\t(\mathrm{d}\mathbf{v}))$ is ergodic. 

Let $H_n(\t,\mathbf{v})$ denote the height of the $n$-th branching point on the geodesic ray $\mathbf{v}$. For every $n\geq 1$, it holds $H_n= \sum_{i=0}^{n-1} H_1\circ S^i$ where $S^i$ stands for the $i$-th iterate of the shift~$S$. 
It follows from the ergodic theorem that $\mathcal{W}\a(\t)\Theta_\al(\mathrm{d}\t)\bar \omega\a_{\t}(\mathrm{d}\mathbf{v})$-a.s.~and thus $\Theta_\al(\mathrm{d}\t)\bar \omega\a_{\t}(\mathrm{d}\mathbf{v})$-a.s. 
\begin{equation}
\label{ergodic1}
\frac{1}{n} H_n \build{\longrightarrow}_{n\to\infty}^{} \int H_1(\t,\mathbf{v})\,\mathcal{W}\a(\t)\,\Theta_\al(\mathrm{d}\t) \bar \omega\a_{\t}(\mathrm{d}\mathbf{v}).
\end{equation}
One calculates as in the proof of Proposition~\ref{prop:unif-meas-inv} to get
\begin{equation}
\label{ergodic11}
\int H_1(\t,\mathbf{v})\,\mathcal{W}\a(\t)\,\Theta_\al(\mathrm{d}\t) \bar \omega\a_{\t}(\mathrm{d}\mathbf{v}) = \int z_{\varnothing}\mathcal{W}\a(\t) \, \Theta_\al(\mathrm{d}\t)= \frac{\al-1}{\al}\,. 
\end{equation}

For a fixed geodesic ray $\mathbf{v}=(v_1,v_2,\ldots)$, we write $\mathbf{x}_{n,\mathbf{v}}=((v_1,\ldots,v_n),H_{n+1}(\t,\mathbf{v}))$ for the $n+1$-st branching point on $\mathbf{v}$, and we set $F_n(\t,\mathbf{v}) \colonequals \log \bar \omega\a_\t(\{\mathbf{u}\in\partial \t \colon \mathbf{x}_{n,\mathbf{v}}\prec\mathbf{u}\})$. 
Using the ergodic theorem again, we have $\Theta_\al(\mathrm{d}\t)\,\bar \omega\a_{\t}(\mathrm{d}\mathbf{v})$-a.s.,
\begin{equation*}
\frac{1}{n} F_n \build{\longrightarrow}_{n\to\infty}^{} \int F_1(\t,\mathbf{v})\,\mathcal{W}\a(\t)\, \Theta_\al(\mathrm{d}\t)\bar \omega\a_{\t}(\mathrm{d}\mathbf{v}),
\end{equation*}
where the limit is identified with 
\begin{equation}
\label{ergodic2}
-\frac{1}{\al-1} \int z_{\varnothing}\mathcal{W}\a(\t) \, \Theta_\al(\mathrm{d}\t) = - \al^{-1}.
\end{equation}
The calculation can be done analogously as in display (18) of~\cite{LIN2}.
We only point out that we need $\int \mathcal{W}\a(\t)|\log \mathcal{W}\a(\t)| \, \Theta_\al(\mathrm{d}\t)< \infty$, which is true because $\sum \theta_\al(k) k(\log k)^2<\infty$ (cf.~Theorem I.10.2 in~\cite{AN}). 
By considering the ratio $F_n/H_n$ and taking $n\to \infty$, we derive that $\Theta_\al(\mathrm{d}\t)$-a.s.~$\bar\omega\a_\t(\mathrm{d}\mathbf{v})$-a.e.,
\begin{displaymath}
\lim_{r\to \infty} \frac{1}{r}\log \bar \omega\a_\t(\mathcal{B}(\mathbf{v},r))=-\frac{1}{\al-1},
\end{displaymath}
from which the convergence (\ref{eq:loc-dim-unif}) readily follows.

Then we turn to the harmonic measure $\mu_\t$ and set $G_n(\t,\mathbf{v}) \colonequals \log \mu_\t(\{\mathbf{u}\in\partial \t\colon \mathbf{x}_{n,\mathbf{v}}\prec\mathbf{u}\})$. 
Using the flow property of $\mu_\t$ (see Lemma 2.3 in~\cite{LIN1}) and the ergodic theorem, we obtain the $\Theta_\al(\mathrm{d}\t)\bar \omega\a_{\t}(\mathrm{d}\mathbf{v})$-almost sure convergence
\begin{equation}
\label{ergodic3}
\frac{1}{n} G_n \build{\longrightarrow}_{n\to\infty}^{} \int G_1(\t,\mathbf{v})\,\mathcal{W}\a(\t)\, \Theta_\al(\mathrm{d}\t)\bar \omega\a_{\t}(\mathrm{d}\mathbf{v}).
\end{equation}
We can evaluate the preceding limit as
\begin{align*}
& \int \sum\limits_{i=1}^{k_\varnothing} e^{-\frac{z_{\varnothing}}{\al-1}} \mathcal{W}\a(\t_{(i)}) \log \frac{\mathcal{C}(\t_{(i)})}{\mathcal{C}(\t_{(1)})+\cdots+\mathcal{C}(\t_{(k_\varnothing)})} \, \Theta_\al(\mathrm{d}\t) \\
=\,\, & \frac{\al-1}{\al} \int k_\varnothing \mathcal{W}\a(\t_{(1)}) \log \frac{\mathcal{C}(\t_{(1)})}{\mathcal{C}(\t_{(1)})+\cdots+\mathcal{C}(\t_{(k_\varnothing)})} \, \Theta_\al(\mathrm{d}\t).
\end{align*}
Putting \eqref{ergodic3} together with (\ref{ergodic1}) and (\ref{ergodic11}), we see that $\Theta_\al(\mathrm{d}\t)$-a.s.~$\bar\omega\a_\t(\mathrm{d}\mathbf{v})$-a.e.,
\begin{equation*}
\lim_{r\to \infty} \frac{1}{r}\log  \mu_\t(\mathcal{B}(\mathbf{v},r))=\int k_\varnothing \mathcal{W}\a(\t_{(1)}) \log \frac{\mathcal{C}(\t_{(1)})}{\mathcal{C}(\t_{(1)})+\cdots+\mathcal{C}(\t_{(k_\varnothing)})} \, \Theta_\al(\mathrm{d}\t).
\end{equation*}
which implies that the convergence (\ref{eq:loc-dim-harm}) holds with the limit
\begin{equation}
\label{eq:lal-defi}
\lal \colonequals \E\Bigg[ N_\al \cw\a_1 \log \frac{\cc\a_1+\cdots+\cc\a_{N_\al}}{\cc\a_1} \Bigg],
\end{equation}
where in the right-hand side we keep the same notation as in \eqref{eq:rde-joint}. 

Finally, it remains to check that $\lal$ defined by \eqref{eq:lal-defi} satisfies $(\al-1)^{-1}<\lal<\infty$. 
Since $\cc\a_1 \geq 1$, an upper bound for $\lal$ is 
\[
\E\Big[ N_\al \cw\a_1 \log \big(1+\cc\a_2+\cdots+\cc\a_{N_\al}\big) \Big], 
\] 
which equals $\E[N_\al \log (1+\cc\a_2+\cdots+\cc\a_{N_\al})]$ by independence and the fact that $\E[\cw\a_1]=1$.
For any fixed $r\in (0,\al-1)$, there exists a positive constant $M<\infty$ such that $\log(1+x)\leq M+x^r$ for every $x>0$. 
For similar reasons as in \eqref{eq:logsum-finite}, we have
\begin{displaymath}
\lal \,\leq\, \sum_{k\geq 2} k\theta_\al(k) \E\Big[\log \big( 1+\cc\a_2+\cdots+\cc\a_k\big)\Big] \leq \frac{\al M}{\al-1}+ \big(\E\big[\cc\a\big]\big)^r \sum_{k\geq 2} \theta_\al(k) k^{1+r} <\infty. 
\end{displaymath}
On the other hand, from \eqref{ergodic11} and the display following \eqref{ergodic3} we also know that
\begin{displaymath}
\lal =\frac{\al}{\al-1} \int e^{-\frac{z_{\varnothing}}{\al-1}} \bigg(\sum\limits_{i=1}^{k_\varnothing} \mathcal{W}\a(\t_{(i)}) \log \frac{\mathcal{C}\a(\t_{(1)})+\cdots+\mathcal{C}\a(\t_{(k_\varnothing)})}{\mathcal{C}\a(\t_{(i)})}\bigg) \Theta_\al(\mathrm{d}\t).
\end{displaymath}
By concavity of the logarithm,
\begin{displaymath}
\sum\limits_{i=1}^{k_\varnothing} \mathcal{W}\a(\t_{(i)}) \log \frac{\mathcal{C}\a(\t_{(1)})+\cdots+\mathcal{C}\a(\t_{(k_\varnothing)})}{\mathcal{C}\a(\t_{(i)})} \geq \sum\limits_{i=1}^{k_\varnothing} \mathcal{W}\a(\t_{(i)}) \log \frac{\mathcal{W}\a(\t_{(1)})+\cdots+\mathcal{W}\a(\t_{(k_\varnothing)})}{\mathcal{W}\a(\t_{(i)})}.
\end{displaymath}
This inequality is strict if and only if for some $i\in \{1,\ldots,k_\varnothing\}$,
\begin{displaymath}
\frac{\mathcal{W}\a(\t_{(i)})}{\mathcal{W}\a(\t_{(1)})+\cdots+\mathcal{W}\a(\t_{(k_\varnothing)})}\neq \frac{\mathcal{C}\a(\t_{(i)})}{\mathcal{C}\a(\t_{(1)})+\cdots+\mathcal{C}\a(\t_{(k_\varnothing)})}.
\end{displaymath}
Since the latter property holds with positive probability under $\Theta_\al(\mathrm{d}\t)$, 
\begin{align*}
\lal \, >&\,\, \frac{\al}{\al-1} \int e^{-\frac{z_{\varnothing}}{\al-1}} \bigg(\sum\limits_{i=1}^{k_\varnothing} \mathcal{W}\a(\t_{(i)}) \log \frac{\mathcal{W}\a(\t_{(1)})+\cdots+\mathcal{W}\a(\t_{(k_\varnothing)})}{\mathcal{W}\a(\t_{(i)})}\bigg) \Theta_\al(\mathrm{d}\t)\\
\,=&\,\, - \frac{\al}{\al-1} \int F_1(\t,\mathbf{v})\,\mathcal{W}\a(\t)\,\Theta_\al(\mathrm{d}\t)\,\bar \omega\a_{\t}(\mathrm{d}\mathbf{v}).
\end{align*}
By (\ref{ergodic2}), the right-hand side of the last display is equal to $\frac{1}{\al-1}$. 
Therefore, we have $\lal> \frac{1}{\al-1}$.

\subsection{Proof of Proposition~\ref{prop:dim-formula}}
\label{sec:lal-monotonicity}
According to its definition \eqref{eq:lal-defi}, 
\begin{equation}
\label{eq:lal-2}
\lal =\frac{\al}{\al-1}\E\Bigg[ \log \frac{\widehat{\mathcal{C}}\a+\cc\a_2+\cdots+\cc\a_{\widehat{N}_\al}}{\widehat{\mathcal{C}}\a} \Bigg].
\end{equation}
Since we can identify the latter with $\E[\widehat{\mathcal{C}}\a]-1$ by taking $g(x)=\log(x)$ in \eqref{eq:g-c-c*}, we get the identity $\lal = \E[\widehat{\mathcal{C}}\a]-1 =\E[\cw\a \cc\a]-1$.

Let us show the monotonicity of $\lal$ to complete the proof of Proposition~\ref{prop:dim-formula}.
Indeed, we will prove that the law $\widehat \gamma_\al$ of the size-biased conductance decreases for $\al\in(1,2]$ in the sense of the usual stochastic (partial) order.
Observe that the function $G(v,n,x, (c_i)_{i\geq 2})$ defined in \eqref{eq:G-defi}~is
\begin{itemize}
\item decreasing with respect to $v\,$,
\item increasing with respect to $n\,$,
\item increasing with respect to $c_i$ for every $i\geq 2$. 
\end{itemize}
Since $\P(V_\al >x)=(1-x)^{\frac{\al}{\al-1}}$ for every $x\in [0,1]$, the random variable $V_\al$ is increasing with respect to $\al\in(1,2]$ for the stochastic order.
Meanwhile, we deduce from \eqref{eq:gene-fct} that
\begin{equation}
\label{eq:bias-gene-fct}
\E\big[ r^{\widehat N_\al}\big]= r-r(1-r)^{\al-1}, \quad \forall r\in (0,1].
\end{equation}
Putting it into the identity
\begin{displaymath}
\sum\limits_{k=1}^\infty r^{k-1}\, \P\big(\widehat N_\al\geq k\big) =\sum\limits_{k=0}^\infty r^k \,\P\big( \widehat N_\al >k \big) = \frac{1-\E[r^{\hat N_\al}]}{1-r}
\end{displaymath}
and recalling that a.s.~$\widehat N_\al \geq 2$, we get 
\begin{displaymath}
\sum\limits_{k=3}^\infty r^{k-2}\, \P\big(\widehat N_\al\geq k\big) = (1-r)^{\al-2} -1 \,,
\end{displaymath}
which implies that for every integer $k\geq 3$, 
\begin{displaymath}
\P\big(\widehat N_\al\geq k\big) = (-1)^k \frac{(\al-2)(\al-3)\cdots(\al-k+1)}{(k-2)!} = \frac{k(k-1)\theta_\al(k)}{\al}.
\end{displaymath}
By taking the derivative with respect to $\al$, we see that the size-biased offspring number $\widehat N_\al$ is decreasing with respect to $\al\in(1,2]$ for the stochastic order. 
Recall that the conductance $\cc\a$ is also decreasing with respect to $\al\in(1,2]$ for the stochastic order (see Section 2.4 of~\cite{LIN1}). 

Now we take $1<\al_1 \leq \al_2\leq 2$. According to the previous discussion, for any fixed $x\geq 1$, 
\[
G\big(V_{\al_2}, \widehat N_{\al_2}, x, (\cc^{(\al_2)}_i)_{i\geq 2}\big) \preceq G\big(V_{\al_1}, \widehat N_{\al_1}, x, (\cc^{(\al_1)}_i)_{i\geq 2}\big),
\] 
which implies that $\widehat{\Phi}_{\alpha_2}(\sigma)\preceq \widehat{\Phi}_{\alpha_1}(\sigma)$  for any probability distribution $\sigma$ on $[1,\infty]$. 
By Proposition~\ref{prop:c*-law}, for any $\al\in (1,2]$, the law $\widehat \gamma_\al$ of $\widehat \cc\a$ has no atom and the iterates of $\widehat{\Phi}_\al(\sigma)$ converge weakly to $\widehat \gamma_\al$. 
As a result, $\widehat \cc^{(\al_2)}$ is stochastically dominated by $\widehat \cc^{(\al_1)}$.

Therefore, we have shown that the law of the size-biased conductance $\widehat \cc\a$ is decreasing with respect to $\al\in(1,2]$ for the stochastic order.
In particular, the first moment $\E[\widehat\cc\a]$ decreases with respect to $\al\in (1,2]$ and so does $\lal$.

\subsection{The asymptotic behavior of $\lambda_{\alpha}$ as $\alpha\downarrow 1$}
We give here the proof of Proposition~\ref{prop:dimension-explosion}. 
First of all, let us collect some facts about the Riemann zeta function
$\zeta(s)=\sum_{n\geq 1} 1/n^s$. It is well-known that $\zeta(s)$ has a simple pole at $s=1$ with residue 1. Then for the derivative
$\zeta'(s)=-\sum_{n\geq 1} \frac{\log(n)}{n^s}$, we have
\begin{equation}
\label{eq:zeta-pole}
\lim_{\alpha\downarrow 1} \,\,(\alpha-1)^2 \zeta'(\alpha)=-1.
\end{equation}
The following inequality for Gamma function ratios (see for example~\cite[page 14]{Art}) is also needed. For every integer $k\geq 2$ and every $0\leq x\leq 1$,
$$(k-1)^x \leq \frac{\Gamma(k+x)}{\Gamma(k)}\leq k^x,$$
One can thus find two postive constants $C_1, C_2$ such that for any $\alpha\in (1,3/2)$ and any $k\geq 2$,
\begin{equation}
\label{eq:theta-uniform}
\theta_\alpha(k)=\frac{\alpha}{\Gamma(2-\alpha)}\frac{\Gamma(k-\alpha)}{\Gamma(k+1)}\in \Big(\frac{C_1}{k^{1+\alpha}},\frac{C_2}{k^{1+\alpha}}\Big).
\end{equation}

Starting from \eqref{eq:lal-2} and the fact that $\widehat{\mathcal{C}}\a , \cc\a \geq 1$, we have
$$\lal \leq \frac{\al}{\al-1}\E\Big[ \log \Big( 1 +\cc\a_2+\cdots+\cc\a_{\widehat{N}_\al}\Big)\Big]\leq \frac{\al}{\al-1}\E\Big[ \log \big(\cc\a_1 +\cc\a_2+\cdots+\cc\a_{\widehat{N}_\al}\big)\Big],$$
where in the last expectation we introduced another copy $\cc\a_1$ of $\cc\a$ independent of all other random variables. 
Let $(U_k)_{k\geq 1}$ be a sequence of independent random variables uniformly distributed over $[0,1]$.
Since each $\cc_k\a$ is dominated by $U_k^{-1}$ for the stochastic partial order (because of \eqref{eq:rde}),
$$ \E\Big[ \log \big(\cc\a_1 +\cc\a_2+\cdots+\cc\a_k \big)\Big]\leq \E\Big[\log\big(U_1^{-1}+U_2^{-1}+\cdots+U_k^{-1}\big)\Big].$$
In~\cite[page 11]{LIN1}, we have already shown that
$$\E\Big[\log\big(U_1^{-1}+U_2^{-1}+\cdots+U_k^{-1}\big)\Big] \leq \Big(2+\frac{k}{k-1}\Big)\log k,$$
from which it follows that
$$ \lal \leq \sum_{k\geq 2} k\,\theta_\alpha(k)\Big(2+\frac{k}{k-1}\Big)\log k.$$
Thus, together with \eqref{eq:theta-uniform} and \eqref{eq:zeta-pole}, we get
$$\limsup_{\alpha\downarrow 1} \,(\alpha-1)^2 \lal \leq \lim_{\alpha\downarrow 1} \,(\alpha-1)^2 \sum_{k\geq 2} \frac{4C_2}{k^{\alpha}}\log k\,=\, 4C_2<\infty.$$

For a lower bound, let us write
\begin{equation}
\label{eq:lal-liminf}
\frac{\alpha-1}{\alpha}\lal= \E\Big[ \log \Big(\widehat{\mathcal{C}}\a+\cc\a_2+\cdots+\cc\a_{\widehat{N}_\al}\Big) \Big] - \E\big[\log \widehat{\mathcal{C}}\a\big].
\end{equation}
Since $\widehat{\mathcal{C}}\a , \cc\a \geq 1$, 
$$\E\Big[ \log \Big(\widehat{\mathcal{C}}\a+\cc\a_2+\cdots+\cc\a_{\widehat{N}_\al}\Big) \Big] \geq \E\Big[\log \widehat{N}_\al\Big]= \frac{\al-1}{\al}\sum_{k\geq 2} k\,\theta_\alpha(k)\log k.$$
Still using \eqref{eq:theta-uniform} and \eqref{eq:zeta-pole}, we see that
$$\liminf_{\alpha\downarrow 1} (\alpha-1)^2 \sum_{k\geq 2} k\,\theta_\alpha(k)\log k \geq C_1,$$
and hence
\begin{equation}
\label{eq:lower-bd1}
\liminf_{\alpha\downarrow 1} \,(\alpha-1) \,\E\Big[ \log \Big(\widehat{\mathcal{C}}\a+\cc\a_2+\cdots+\cc\a_{\widehat{N}_\al}\Big) \Big] \geq C_1>0.
\end{equation}
On the other hand, let $V_\alpha\in (0,1)$ be a random variable of density function $\frac{\al}{\al-1}(1-x)^{\frac{1}{\al-1}}$. Due to \eqref{eq:c*-rde}, the size-biased conductance $\widehat{\mathcal{C}}\a$ is dominated by $V_\alpha^{-1}$ for the stochastic order, and thus $\E[\log \widehat{\mathcal{C}}\a]\leq \E[\log V_\alpha^{-1}]$. 
However, we define the analytic function
$$H(x)\colonequals \int_0^1 \frac{1-t^x}{1-t} \mathrm{d}t, \quad \mbox{ for } x\geq 1,$$
that interpolates the harmonic numbers $H(n)=\sum_{k=1}^n 1/k, n\geq 1$. An integration by parts shows that 
$$H\Big(\frac{\alpha}{\alpha-1}\Big)=-\frac{\alpha}{\alpha-1} \int_0^1\log(x) (1-x)^{\frac{1}{\al-1}}\,\mathrm{d}x=\E\Big[\log V_\alpha^{-1}\Big].$$
Since $H(x)-\log(x)$ converges to the famous Euler-Mascheroni constant as $x$ goes to $+\infty$,  
we have 
$$\lim_{\alpha\downarrow 1} \,(\alpha-1) \E\Big[\log \widehat{\mathcal{C}}\a\Big]=\lim_{\alpha\downarrow 1} \,(\alpha-1) \E\Big[\log V_\alpha^{-1}\Big]=0.$$
Combining this with \eqref{eq:lal-liminf} and \eqref{eq:lower-bd1}, we obtain 
$$\liminf_{\alpha\downarrow 1} \,(\alpha-1)^2 \lal \geq C_1>0.$$
The proof of Proposition~\ref{prop:dimension-explosion} is therefore completed.

\section{The discrete setting}
\label{sec:discrete}
\renewcommand{\t}{\mathsf{T}}

\subsection{Galton--Watson trees}
\label{sec:tree-discrete}
We briefly introduce the general notion of discrete rooted ordered trees, which extends the case of infinite trees without leaves that we have seen in Section~\ref{sec:treedelta}.
A discrete rooted ordered tree~$\mathsf{t}$ is by definition a subset of $\mathcal{V}$ such that the following holds:
\begin{enumerate}
\item[(i)] $\varnothing\in \mathsf{t}\,$;

\item[(ii)] If $u=(u_1,\ldots,u_n)\in \mathsf{t} \backslash\{\varnothing\}$, then $\bar u=(u_1,\ldots,u_{n-1})\in \mathsf{t}\,$;

\item[(iii)] For every $u=(u_1,\ldots,u_n)\in \mathsf{t}$, there exists an integer $k_u(\mathsf{t})\geq 0$ such that, for every $j\in\N$, $(u_1,\ldots,u_n,j)\in \mathsf{t}$ if and only if $1\leq j\leq k_u(\mathsf{t})$.
\end{enumerate}
The quantity $k_u(\mathsf{t})$ above is called the number of children of~$u$ in $\mathsf{t}$. A vertex with no child is called a leaf. From now on, we will just say tree instead of discrete rooted ordered tree for short. 
We write~$\prec$ for the (non-strict) genealogical order on $\mathsf{t}$. 
A tree $\mathsf{t}$ is always viewed as a graph whose vertices are the elements of $\mathsf{t}$ and whose edges are the pairs $\{\bar u,u\}$ for all $u\in  \mathsf{t}\backslash\{\varnothing\}$. 
Since the lexicographical order on the vertices is not really used in our arguments, we will not pay much attention to this order structure.  

For an infinite tree $\mathsf{t}$, we say it has a single infinite line of descent if there exists a unique sequence of positive integers $(u_n)_{n\geq 1}$ such that $(u_1,u_2,\ldots,u_n)\in \mathsf{t}$ for all $n\geq 1$. 

The height of a tree $\mathsf{t}$ is written as $h(\mathsf{t})\colonequals \sup\{|u|\colon u\in \mathsf{t}\}$. The set of all vertices of $\mathsf{t}$ at generation~$n$ is denoted by $\mathsf{t}_n\colonequals \{u \in \mathsf{t} \colon |u|=n\}$. If $u\in \mathsf{t}$, the tree of descendants of $u$ is $\mathsf{t}[u]\colonequals \{w\in \mathcal{V}\colon uw\in \mathsf{t}\}$.

Let $\mathsf{t}$ be a tree of height larger than $n$, and consider a simple random walk $X=(X_k)_{k\geq 0}$ on $\mathsf{t}$ starting from the root $\varnothing$, which is defined under the probability measure~$P^{\mathsf{t}}$. We write $\tau_n\colonequals \inf\{k\geq 0 \colon |X_k|=n\}$ for the first hitting time of generation $n$ by $X$, and we define the discrete harmonic measure $\mu_n^{\mathsf{t}}$ supported on $\mathsf{t}_n$ as the law of $X_{\tau_n}$ under $P^{\mathsf{t}}$.

\smallskip
{\bf Critical Galton--Watson trees.}
For every integer $n\geq 0$, we let $\t^{(n)}$ be a discrete Galton--Watson tree with critical offspring distribution $\rho$, conditioned on non-extinction at generation~$n$, viewed as a random subset of $\mathcal V$. We suppose that the random trees $\t^{(n)}, n\geq 0$ are defined under the probability measure $\P$. 

We let $\t^{*n}$ be the reduced tree associated with $\t^{(n)}$, consisting of all vertices of $\t^{(n)}$ that have (at least) one descendant at generation $n$. Note that $|u|\leq n$ for every $u\in \t^{*n}$. 
When talking about a reduced tree, we always implicitly assume that the correct relabeling of the vertices has been done to make it a tree without changing the genealogical order. 

Set $q_{n}\colonequals \P\big(h(\t^{(0)})\geq n\big)$ the non-extinction probability up to generation $n$ for a Galton--Watson tree of offspring distribution $\rho$. 
If $L$ is the slowly varying function appearing in (\ref{eq:stable-attraction}), it has been established in~\cite[Lemma~2]{S68} that
\begin{equation}
\label{eq:survivalpro}
q_{n}^{\alpha-1}L(q_{n})\sim \frac{1}{(\alpha-1)n} \quad \mbox{ as } n \to \infty.
\end{equation}
By the asymptotic inversion property of slowly varying functions (see e.g.~\cite[Section 1.5.7]{BGT87}), it follows that
\begin{equation*}
q_{n}\sim n^{-\frac{1}{\alpha-1}}\ell(n) \quad \mbox{ as } n \to \infty,
\end{equation*} 
for some other function $\ell$ slowly varying at $\infty$. 

{\bf Size-biased Galton--Watson tree.} 
We denote by $\widehat{\mathsf{T}}$ a size-biased Galton--Watson tree with offspring distribution $\rho$.
It is defined similarly as the size-biased CTGW tree $\widehat \Gamma\a$. 
We refer to Lyons, Pementle and Peres~\cite{LPP95b} for the details. 
The unique infinite line of descent $(\mathbf{v}_1,\mathbf{v}_2,\ldots)$ in $\widehat{\mathsf{T}}$ is called its spine.
If $\widehat N$ denotes a random variable distributed according to the size-biased distribution of $\rho$, that is, $\P(\widehat N=k)=k\,\rho(k)$ for every $k\geq 0$, 
then the offspring numbers of $(\mathbf{v}_k)_{k\geq 1}$ are i.i.d.~copies of $\widehat N$. 

Let $[\widehat{\mathsf{T}}]^{(n)}$ be the finite tree obtained from $\widehat{\mathsf{T}}$ by keeping only its first $n$ generations, i.e.,
$$[\widehat{\mathsf{T}}]^{(n)} \colonequals \{u\in \widehat{\mathsf{T}}\colon |u|\leq n\}.$$
As shown in~\cite{LPP95b}, the law of the random tree $[\widehat{\mathsf{T}}]^{(n)}$ is that of $\mathsf{T}^{(n)}$ biased by $\#\mathsf{T}^{(n)}_n$. Moreover, conditionally given the first $n$ levels of $\widehat{\mathsf{T}}$, the vertex $\mathbf{v}_n$ on the spine is uniformly distributed on the $n$-th level of $\widehat{\mathsf{T}}$. 

For every integer $n\geq 1$, let $[\widehat{\mathsf{T}}]^n$ be the finite tree obtained from $\widehat{\mathsf{T}}$ by erasing the (infinite) tree of descendants of the vertex $\mathbf{v}_n$. By convention, the vertex $\mathbf{v}_n$ itself is kept in $[\widehat{\mathsf{T}}]^n$. 
We emphasize that in general $[\widehat{\mathsf{T}}]^n \neq [\widehat{\mathsf{T}}]^{(n)}$, since the height of $[\widehat{\mathsf{T}}]^n$ can be strictly larger than $n$.

We let $[\widehat{\mathsf{T}}]^{*n}$ be the reduced tree associated with the tree $[\widehat{\mathsf{T}}]^n$ up to generation~$n$, which consists of all vertices of $[\widehat{\mathsf{T}}]^n$ that have (at least) one descendant at generation $n$. Notice that $[\widehat{\mathsf{T}}]^{*n}$ is also the reduced tree associated with $[\widehat{\mathsf{T}}]^{(n)}$ up to generation $n$.

\subsection{Convergence of discrete conductances}
\label{sec:cv-dis-cond}
Take a tree $\mathsf{t}$ of height larger than $n$ and consider the new tree $\mathsf{t}'$ obtained by adding to the graph $\mathsf{t}$ an edge between its root $\varnothing$ and an extra vertex $\bar\varnothing$. 
We define under the probability measure $P^{\mathsf{t}'}$ a simple random walk $X$ on $\mathsf{t}'$ starting from the root $\varnothing$. 
Let $\tau_{\bar\varnothing}$ be the first hitting time of $\bar\varnothing$ by $X$, and for every integer $1\leq i \leq n$, let $\tau_i$ be the first hitting time of generation $i$ (of the tree $\mathsf{t}$) by $X$. 
We write
$$\mathcal{C}_i(\mathsf{t})\colonequals P^{\mathsf{t}'}(\tau_i<\tau_{\bar\varnothing}).$$
This notation is justified by the fact that $\mathcal{C}_i(\mathsf{t})$ can be interpreted as the effective conductance between $\bar\varnothing$ and the $i$-th generation of $\mathsf{t}$ in the graph $\mathsf{t}'$, see e.g.~\cite[Chapter 2]{LP10}.

Here are some basic facts concerning the conductance of the reduced Galton--Watson tree~$\mathsf{T}^{*n}$. 

\begin{proposition}
\label{prop:cv-conductance}
Suppose that the critical offspring distribution $\rho$ is in the domain of attraction of a stable distribution of index $\al\in(1,2]$. Then it holds that 
\begin{displaymath}
n\,\mathcal{C}_n(\mathsf{T}^{*n}) \xrightarrow[n\to\infty]{(\mathrm{d})}  \cc\a.
\end{displaymath}
\end{proposition}

Recall that $\cc\a$ is the conductance of the continuous reduced tree $\Delta\a$. 
Given the convergence result of rescaled discrete trees $n^{-1}\mathsf{T}^{*n}$ towards $\Delta\a$ (Proposition 3.2 in \cite{LIN1}), the previous proposition can be shown in the same way as Proposition~11 in~\cite{LIN2}. 

The following moment estimate for the conductance $\mathcal{C}_n(\mathsf{T}^{*n})$ is Lemma 3.9 in~\cite{LIN1}. 
\begin{lemma}
\label{lem:Lr-gw-condct}
For every $r\in(0,\al)$, there exists a constant $K=K(r, \rho)\geq 1$ depending on $r$ and the offspring distribution $\rho$ such that, for every integer $n\geq 1$,
\begin{displaymath}
\E\Big[\big(n\,\mathcal{C}_n(\mathsf{T}^{*n})\big)^r\Big]\leq K.
\end{displaymath}
\end{lemma}

\subsection{Backward size-biased Galton--Watson tree}
\label{sec:backward-size-biased}

In the proof of Theorem~\ref{thm:dim-discrete}, we will need a rear-view variant $\widecheck{\mathsf{T}}$ of the size-biased Galton--Watson tree $\widehat{\mathsf{T}}$, which can be tracked back to the inflated Galton--Watson tree introduced by Peres and Zeitouni in~\cite{PZ08}.
Let us explain succinctly how this backward size-biased Galton--Watson tree $\widecheck{\mathsf{T}}$ is constructed. 
More details can be found in Section 3.4 of~\cite{LIN2}. 

First, the random tree $\widecheck{\mathsf{T}}$ has a unique infinite ray of vertices $(\mathbf{u}_0,\mathbf{u}_1,\mathbf{u}_2,\ldots)$, referred to as its spine. 
We fix a genealogical order on the spine by declaring that, for every $n\geq 0$, the vertex $\mathbf{u}_n$ is at generation $-n$ of $\widecheck{\mathsf{T}}$.
Under this rule, $\mathbf{u}_{0}$ is a child of $\mathbf{u}_1$, which is a child of $\mathbf{u}_2$, etc.

To describe the finite subtrees in $\widecheck{\mathsf{T}}$ branching off every node of the spine, we take an i.i.d.~sequence $(\widehat N_n)_{n\geq 1}$ of random variables following the size-biased distribution of $\rho$.
Independently to each vertex $\mathbf{u}_n, n\geq 1$, we graft $\widehat N_n-1$ independent copies of an ordinary $\rho$-Galton--Watson descendant tree, so that $\mathbf{u}_n$ has exactly $\widehat N_n$ children: one of them, $\mathbf{u}_{n-1}$, is on the spine while the rest of them are the roots of independent $\rho$-Galton--Watson trees. 

Last but not least, we define the genealogical order on $\widecheck{\mathsf{T}}$ by keeping the genealogical orders inherited from the grafted Galton--Watson trees and combining them with the genealogical order on the spine. 
For instance, $\mathbf{u}_2$ is an ancestor of any vertex in one of the subtrees grafted at $\mathbf{u}_1$. 
The notion of generation for every vertex in $\widecheck{\mathsf{T}}$ can also be defined in a consistent manner: 
for any vertex $v$ not on the spine, there is a unique vertex $\mathbf{u}_{m}$ on the spine such that $v$ belongs to a finite subtree grafted at $\mathbf{u}_{m}$, then we say that the generation of $v$ in $\widecheck{\mathsf{T}}$ is equal to $-m+1$ plus the initial generation of $v$ inside the corresponding grafted tree. See Figure~\ref{figure:backgwbias} for an illustration. 

\begin{figure}[!h]
\begin{center}
\includegraphics[width=10cm]{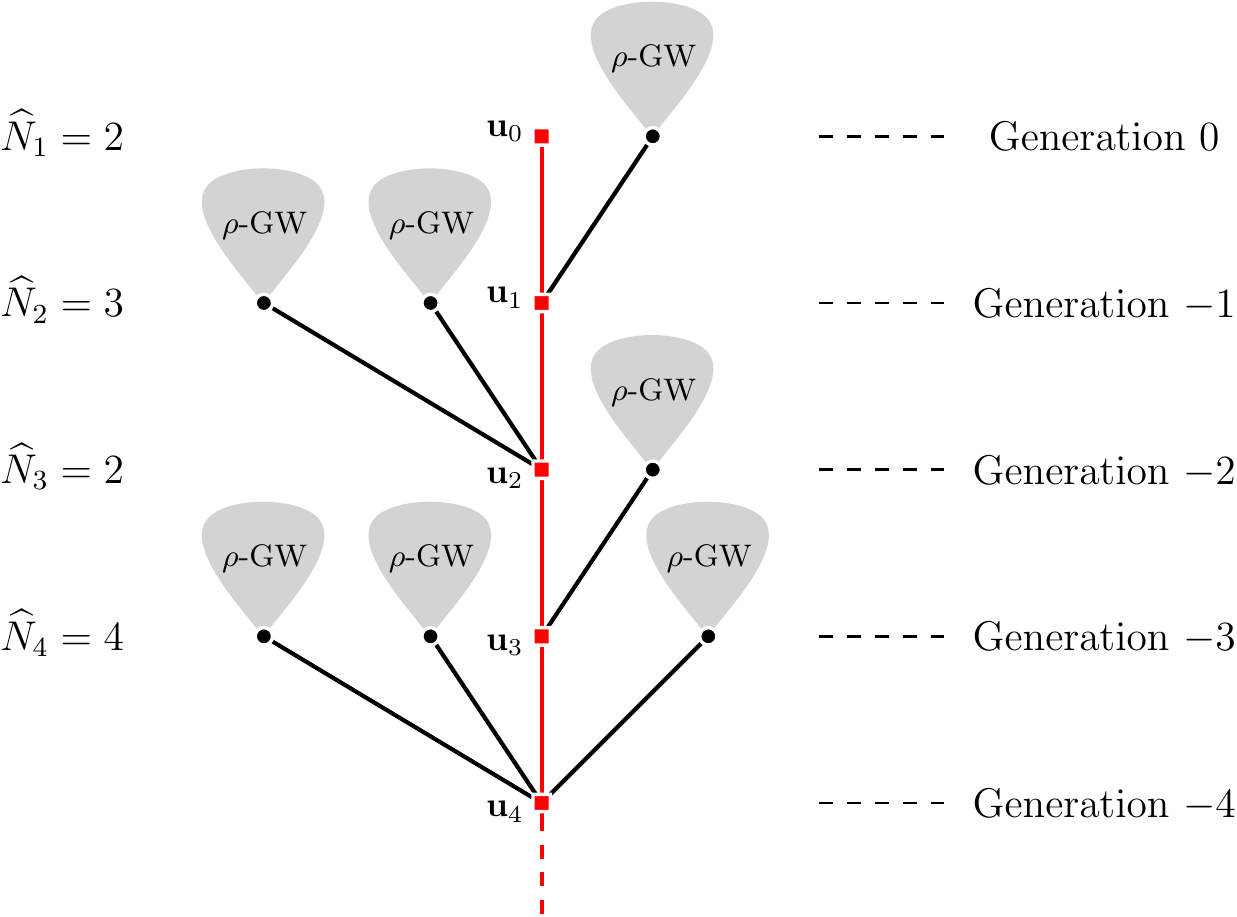}
\caption{Schematic representation of the backward size-biased Galton--Watson tree~$\check{\mathsf{T}}$ \label{figure:backgwbias}}
\end{center}
\end{figure}

Different from the inflated Galton--Watson tree described in Peres and Zeitouni~\cite{PZ08}, the vertex $\mathbf{u}_0$ in $\widecheck{\mathsf{T}}$ has no strict descendant. 

For every $n\geq 1$, let $[\widecheck{\mathsf{T}}]^{n}$ be the tree obtained from $\widecheck{\mathsf{T}}$ by only keeping the finite tree above~$\mathbf{u}_n$. 
The vertex $\mathbf{u}_n$ is taken as the root of $[\widecheck{\mathsf{T}}]^{n}$.
We observe that $[\widecheck{\mathsf{T}}]^{n}$ has the same distribution as the random tree $[\widehat{\mathsf{T}}]^{n}$ defined at the end of Section~\ref{sec:tree-discrete}. 
Moreover, the root $\mathbf{u}_n$ of $[\widecheck{\mathsf{T}}]^{n}$ corresponds to the root $\varnothing$ of $[\widehat{\mathsf{T}}]^{n}$, and the vertex $\mathbf{u}_0$ in $[\widecheck{\mathsf{T}}]^{n}$ corresponds to the vertex $\mathbf{v}_n$ in $[\widehat{\mathsf{T}}]^{n}$.

\section{Proof of Theorem~\ref{thm:dim-discrete}}
\label{sec:proof-thm1}

To prove Theorem~\ref{thm:dim-discrete}, we will follow the same idea developed in~\cite{LIN2} for the finite variance case: 
the problem can be reduced to calculating a certain hitting probability for simple random walk on the backward size-biased Galton--Watson tree $\widecheck{\mathsf{T}}$ (see Proposition~\ref{prop:red-theorem} below). Since we fix a critical offspring distribution $\rho$ satisfying \eqref{eq:stable-attraction}, the parameter $\al\in(1,2]$ is always fixed in this section.
For simplicity, we assume that all the random trees involved below are defined under the same probability measure $\P$.

Recall that $\lambda_\al=\E[\widehat \cc\a]-1>\frac{1}{\al-1}$ is the constant that appears in~\eqref{eq:loc-dim-harm}. 
Our first reduction goes from the conditional Galton--Watson tree $\mathsf{T}^{(n)}$ to the size-biased Galton--Watson tree $\widehat{\mathsf{T}}$. 
Using a result of Slack~\cite[Theorem 1]{S68}, one can easily adapt the arguments at the beginning of Section 4 in~\cite{LIN2} to see that convergence~\eqref{eq:dim-discrete} holds if for every $\delta>0$, 
\begin{equation*}
\lim_{n\to \infty}  \P\Big(n^{-\lambda_\al-\delta}\leq P^{[\widehat{\mathsf{T}}]^{n}}(X_{\tau_n}=\mathbf{v}_n)\leq n^{-\lambda_\al+\delta} \Big) =1.
\end{equation*}
Furthermore, owing to the correspondence between the truncated trees $[\widehat{\mathsf{T}}]^{n}$ and $[\widecheck{\mathsf{T}}]^{n}$, we can rewrite the previous display as 
\begin{equation}
\label{eq:1}
\lim_{n\to \infty}  \P\Big(n^{-\lambda_\al-\delta}\leq P^{[\check{\mathsf{T}}]^n}(X_{\tau_n}=\mathbf{u}_0)\leq n^{-\lambda_\al+\delta} \Big) =1.
\end{equation}

\subsection{SRW on the infinite backward tree}
\label{sec:pb-inf}

In order to show \eqref{eq:1}, we denote by $-M_1,-M_2,\ldots$ the generations of the vertices on the spine of $\widecheck{\mathsf{T}}$ where there is at least one grafted tree that has a descendant of generation 0. 
This sequence of negative integers $(-M_k)_{k\geq 1}$ is listed in the strict decreasing order, and we set by convention $M_0=0$. 
For every $k\geq 1$, we also set $L_k \colonequals M_k-M_{k-1}\geq 1$.

For every $n\geq 1$, let $k_n \colonequals k_n(\widecheck{\mathsf{T}})$ be the index such that $M_{k_n}\leq n< M_{k_n+1}$. 
We need the next result to study the asymptotic behavior of $k_n$.

\begin{lemma}
\label{lem:size-bias-gene}
Let $\widehat N$ be a random variable distributed according to the size-biased distribution of~$\rho$. Then there exists a function $\widetilde L$ slowly varying at $0^+$ such that for every $r\in(0,1)$, 
\begin{equation}
\label{eq:size-bias-gene}
\E\big[r^{\widehat N-1}\big]=1-\al (1-r)^{\al-1}\widetilde L(1-r). 
\end{equation}
Moreover, we have
\begin{equation}
\label{eq:asymp-equiv}
\lim_{x\to 0^+} \frac{\widetilde L(x)}{L(x)}=1,
\end{equation}
where $L$ is the slowly varying function appearing in \eqref{eq:stable-attraction}.
\end{lemma}

\begin{proof}
We need a nice Karamata's representation of the slowly varying function $L$ appearing in \eqref{eq:stable-attraction}. Indeed, $L$ can be expressed as
\[
L(x)=C \exp \Big(\int_1^x \frac{\varepsilon(t)}{t}\mathrm{d}t\Big)\,, \qquad x\in(0,1),
\]
where $C\neq 0$ is a constant and $\varepsilon(\cdot)$ is a continuous function on $(0,1)$ such that $\varepsilon(x)\to 0$ as $x\to 0^+$. 
To see this, we follow Slack~\cite{S68} to set $x=1-r$ and $\Lambda(x)= x^{\al-1} L(x)$.
Since $\sum k\rho(k)=1$ and $\rho(1)\neq 1$, we know that
\[ 
x\Lambda(x)= x^\al L(x)=\sum_{k\geq 0}\rho(k) r^k -r
\]
is strictly positive. From the latter expression, we also deduce that
$x\Lambda(x)$ is analytic and has monotone derivative for $x\in(0,1)$. 
It follows from a result of Lamperti~\cite[Theorem 2]{Lam} that 
\[
\lim_{x\to 0^+} \frac{x(x\Lambda(x))'}{x\Lambda(x)}=\al \quad \mbox{and thus} \quad \lim_{x\to 0^+} \frac{x\Lambda'(x)}{\Lambda(x)}=\al-1\,.
\]
We define 
\[
\varepsilon (x) \colonequals \frac{x\Lambda'(x)}{\Lambda(x)}-\al+1\,,
\]
which is continuous for $x\in (0,1)$ and satisfies that $\varepsilon(x)\to 0$ as $x\to 0^+$.
Integrating 
\[
\frac{\Lambda'(x)}{\Lambda(x)}=\frac{\al-1}{x} + \frac{\varepsilon(x)}{x}\,,
\]
we obtain 
\[
\Lambda(x)=C x^{\al-1} \exp\Big(\int_1^x \frac{\varepsilon(t)}{t} \mathrm{d}t\Big)\,,
\]
where $C\neq 0$ as $\rho(1)\neq 1$. The claimed representation for $L$ readily follows.

As an immediate consequence, $L'(x)=L(x)\frac{\varepsilon(x)}{x}$ for $x\in(0,1)$.
By differentiating \eqref{eq:stable-attraction}, we see that
\begin{align*}
\E\big[r^{\widehat N-1}\big] =\sum_{k\geq 1} k\rho(k) r^{k-1} & = 1-\al(1-r)^{\al-1} L(1-r)- (1-r)^{\al} L(1-r)\frac{\varepsilon(1-r)}{1-r} \\
& = 1-\al(1-r)^{\al-1} \widetilde L(1-r)\,,
\end{align*}
where we have put $\widetilde L(1-r)\colonequals L(1-r)(1+\frac{\varepsilon(1-r)}{1-r})$.
When $r\to 1^-$, we get~\eqref{eq:asymp-equiv} as $\varepsilon(1-r)\to 0$. 
The asymptotic equivalence of $\widetilde L$ and $L$ implies that $\widetilde L$ also varies slowly at $0^+$.
\end{proof}

\begin{lemma}
\label{lem:kn-asymptotic}
We have $\P$-a.s.
\begin{displaymath}
\lim_{n\to\infty} \frac{k_n}{\log n} = \frac{\al}{\al-1}.
\end{displaymath}
\end{lemma}

\begin{proof}
Recall that for every $j\geq 1$, there are $\widehat N_j-1$ independent Galton--Watson trees grafted at $\mathbf{u}_j$ in $\widecheck{\mathsf{T}}$. 
Consider the event that at least one of those trees grafted at $\mathbf{u}_j$ reaches generation~0, and let $\epsilon_j$ be the corresponding indicator function. 
Then by definition, 
\begin{displaymath}
\P(\epsilon_j=0)=\E\Big[(1-q_{j-1})^{\widehat N_j-1}\Big]=\E\Big[(1-q_{j-1})^{\widehat N-1}\Big].
\end{displaymath}
Applying \eqref{eq:survivalpro} and Lemma~\ref{lem:size-bias-gene} to the latter formula yields 
\begin{equation}
\label{eq:epsilon-estimate}
\P(\epsilon_j=0)= 1-\frac{\al}{(\al-1)j}+o\big(\frac{1}{j}\big)\,, \qquad \mbox{ as } j\to \infty.
\end{equation}
Since $k_n=\epsilon_1+\epsilon_2+\cdots+\epsilon_n$, we deduce that
\begin{displaymath}
\E[k_n]=\sum\limits_{j=1}^n \big(1-\P(\epsilon_j=0)\big)\, \sim \, \frac{\al}{\al-1}\log n\,, \qquad \mbox{as } n\to \infty.
\end{displaymath}
Since $\epsilon_1, \ldots,\epsilon_n$ are independent, we also have $\mathrm{var}(k_n)=O(\log n)$, and the $L^2$-convergence of $k_n/\log n$ to $\frac{\al}{\al-1}$ follows immediately. The a.s.~convergence is then obtained by standard monotonicity and Borel--Cantelli arguments.
\end{proof}

Given $\widecheck{\mathsf{T}}$, for every $j \geq 0$ we write $P^{\check{\mathsf{T}}}_j$ for the (quenched) probability measure under which we consider a simple random walk $X=(X_k)_{k\geq 0}$ on $\widecheck{\mathsf{T}}$ starting from the vertex $\mathbf{u}_j$. 
Under $P^{\check{\mathsf{T}}}_j$, we denote by $S_0$ the hitting time of generation 0 by the simple random walk $X$, and for every $i \geq 0$, we write $\Pi_i \colonequals \inf\{k\geq 0\colon X_k=\mathbf{u}_i\}$ for the hitting time of vertex $\mathbf{u}_i$.

The following extension of~\cite[Proposition 14]{LIN2} compares simple random walks on the truncated finite tree $[\widecheck{\mathsf{T}}]^n$ and on the backward infinite tree $\widecheck{\mathsf{T}}$.
We skip its proof because one can easily adapt the original proof of Proposition 14 in~\cite{LIN2} to the present setting. 

\begin{proposition}
\label{prop:red-to-backward}
For every $\delta>0$, there exists an integer $n_0\in \N$ such that for every $n\geq n_0$, we have
\begin{displaymath}
\P\Big(P^{[\check{\mathsf{T}}]^n}(X_{\tau_n}=\mathbf{u}_0)\geq n^{-\lambda_\al+\delta}\Big) \leq 2^{\frac{2\al}{\al-1}}\, \P\Big(P^{\check{\mathsf{T}}}_{M_{k_n}}(X_{S_0}=\mathbf{u}_0, S_0<\Pi_{M_{k_n+1}})\geq n^{-\lambda_\al+\delta}/2 \Big),
\end{displaymath}
and
\begin{displaymath}
\P\Big(P^{[\check{\mathsf{T}}]^n}(X_{\tau_n}=\mathbf{u}_0)\leq n^{-\lambda_\al-\delta}\Big) \leq 2^{\frac{2\al}{\al-1}}\, \P\Big(P^{\check{\mathsf{T}}}_{M_{k_n}}(X_{S_0}=\mathbf{u}_0, S_0<\Pi_{M_{k_n+1}})\leq n^{-\lambda_\al-\delta} \Big).
\end{displaymath}
\end{proposition}

Combining (\ref{eq:1}) with the previous result, we can derive Theorem~\ref{thm:dim-discrete} from the next proposition.
\begin{proposition}
\label{prop:red-theorem}
For every $\delta>0$, it holds that
\begin{equation}
\label{eq:red-theorem}
\lim_{n\to \infty}  \P\Big(n^{-\lambda_\al-\delta}\leq P^{\check{\mathsf{T}}}_{M_{k_n}}\!(X_{S_0}=\mathbf{u}_0, S_0<\Pi_{M_{k_n+1}}) \leq n^{-\lambda_\al+\delta}  \Big) =1.
\end{equation}
\end{proposition}

Under the probability measure $\P$, we set therefore
\begin{displaymath}
p_k=p_k(\widecheck{\mathsf{T}}) \colonequals P^{\check{\mathsf{T}}}_{M_k}(X_{S_0}=\mathbf{u}_0, S_0<\Pi_{M_{k+1}})
\end{displaymath}
for every $k\geq 1$.
By the definition of $M_k$, there exists at least one subtree grafted to $\mathbf{u}_{M_k}$ that reaches generation~0. 
The root of this subtree is necessarily a child of $\mathbf{u}_{M_k}$ distinct from $\mathbf{u}_{M_k-1}$. 
If such a subtree is unique, we let $c_k=c_k(\widecheck{\mathsf{T}})$ be the probability that a simple random walk starting from its root reaches generation 0 before hitting $\mathbf{u}_{M_k}$. 
If there is more than one such grafted trees, we take $c_k$ to be the sum of the corresponding probabilities. 
We justify this definition by the fact that $c_k$ can be interpreted as the effective conductance between $\mathbf{u}_{M_k}$ and generation 0 in the graph that consists only of the vertex $\mathbf{u}_{M_k}$ and all the subtrees grafted to it.

We also set, for every $k\geq 1$,
\begin{displaymath}
h_k=h_k(\widecheck{\mathsf{T}}) \colonequals P^{\check{\mathsf{T}}}_{M_k-1}(S_0<\Pi_{M_k}),
\end{displaymath}
which is the probability that a simple random walk starting from $\mathbf{u}_{M_k-1}$ reaches generation 0 before hitting $\mathbf{u}_{M_k}$. 
We write in addition $\ell_k=1/L_k= (M_k-M_{k-1})^{-1}$. 
A recurrence relation between $p_k$ and $p_{k-1}$ is obtained in Section~4.1 of~\cite{LIN2}, which states that 
\begin{displaymath}
p_1=p_k \times \prod_{j=2}^k\Big( 1+ \frac{c_j+\ell_{j+1}}{\ell_j}-\frac{\ell_j}{\ell_j+c_{j-1}+h_{j-1}}\Big).
\end{displaymath}
Thus we define, for every $j\geq 2$,
\begin{displaymath}
Q_j=Q_j(\widecheck{\mathsf{T}}) \colonequals  \log \Big( 1+\frac{c_j+\ell_{j+1}}{\ell_j}-\frac{\ell_j}{\ell_j+c_{j-1}+h_{j-1}}\Big).
\end{displaymath}

\begin{lemma}
\label{lem:sum-L2-conv}
We have
\begin{equation}
\label{eq:sum-L2-conv}
\frac{1}{k} \sum\limits_{j=2}^k Q_j  \,\xrightarrow[k\to\infty]{L^2(\P)} \, \frac{\al-1}{\al}\lambda_\al.
\end{equation}
\end{lemma}

Proposition~\ref{prop:red-theorem} (and thus Theorem~\ref{thm:dim-discrete}) can be easily deduced from this key lemma.
In fact, on account of Lemma~\ref{lem:kn-asymptotic} and Lemma~\ref{lem:sum-L2-conv}, for any $\delta>0$, the event
\begin{displaymath}
\Big\{(\lambda_\al-\delta/2)\log n\leq \sum\limits_{j=2}^{k_n} Q_j \leq (\lambda_\al+\delta/2)\log n \Big\}=\Big\{ p_1 n^{-\lambda_\al-\delta/2}\leq p_{k_n}\leq p_1 n^{-\lambda_\al+\delta/2}  \Big\}
\end{displaymath}
holds with $\P$-probability tending to~1 as $n\to \infty$. Since for sufficiently large $n$, the last event is included in 
\begin{displaymath}
\Big\{ n^{-\lambda_\al-\delta}\leq p_{k_n}\leq n^{-\lambda_\al+\delta}  \Big\}= \Big\{n^{-\lambda_\al-\delta}\leq P^{\check{\mathsf{T}}}_{M_{k_n}}\!(X_{S_0}=\mathbf{u}_0, S_0<\Pi_{M_{k_n+1}}) \leq n^{-\lambda_\al+\delta}\Big\},
\end{displaymath}
the required convergence (\ref{eq:red-theorem}) follows immediately. 
Therefore, we are left with the task of proving Lemma~\ref{lem:sum-L2-conv}.

\subsection{Proof of Lemma~\ref{lem:sum-L2-conv}}

Let us begin with several auxiliary results.
For every $k\geq 1$, we denote by $I_k$ the number of independent $\rho$-Galton--Watson trees grafted at $\mathbf{u}_k$ that reach generation 0 in $\widecheck{\mathsf{T}}$. 

\begin{lemma}
\label{lem:nb-graft}
As $k$ tends to $\infty$, the conditional generating function $\E[r^{I_k} \!\mid\! I_k\geq 1]$ converges to $1-(1-r)^{\al-1}$. 
Thus according to \eqref{eq:bias-gene-fct}, the conditional distribution of $I_k$ given $I_k\geq 1$ converges in distribution to $\widehat N_\al-1$. 
\end{lemma}

\begin{proof}
For fixed $r\in (0,1)$, by independence of the $\widehat N_k -1$ subtrees grafted at $\mathbf{u}_k$, we have
$$ \E\big[r^{I_k}\big] =\E\Big[(rq_{k-1}+1-q_{k-1})^{\widehat N_k -1}\Big]=\E\Big[(rq_{k-1}+1-q_{k-1})^{\widehat N -1}\Big]. $$
It follows from Lemma \ref{lem:size-bias-gene} that
\begin{equation}
\label{eq:Ik-gene}
\E\big[r^{I_k}\big]=1-\al(q_{k-1}(1-r))^{\al-1}\widetilde{L}(q_{k-1}(1-r)),
\end{equation}
where $\widetilde L$ is the slowly varying function appearing in~\eqref{eq:size-bias-gene}.
Thus we obtain
\begin{eqnarray*}
\E\big[r^{I_k} \!\mid\! I_k\geq 1\big] &=& \frac{\E\big[r^{I_k}\big]-\P(I_k=0)}{\P(I_k\geq 1)} \\
  &=& 1- \al(1-r)^{\al-1} \frac{(q_{k-1})^{\al-1}\widetilde L(q_{k-1}(1-r))}{\P(I_k\geq 1)}.
\end{eqnarray*}
Recall that it has been observed in the proof of Lemma~\ref{lem:kn-asymptotic} that
$$ \P(I_k\geq 1)=\P(\epsilon_k=1)=\frac{\al}{(\al-1)k}+o\Big(\frac{1}{k}\Big), \qquad \mbox{ as } k\to \infty. $$
Using \eqref{eq:survivalpro} and \eqref{eq:asymp-equiv}, one readily verifies that 
$$\frac{(q_{k-1})^{\al-1}\widetilde L(q_{k-1}(1-r))}{\P(I_k\geq 1)} \build{\longrightarrow}_{k \to\infty}^{} \frac{1}{\al}\,.$$
The proof is therefore finished. 
\end{proof}

Note that if $\rho$ has a finite variance, the previous lemma becomes trivial because $\P$-a.s.~for all sufficiently large $k$, there will be a unique subtree grafted at $\mathbf{u}_{M_k}$ that reaches generation 0. 
   
With some extra efforts, we also get the following estimate for $I_k$ conditionned on $I_k\geq 1$. 

\begin{lemma}
\label{lemma:I_k-frac-moments}
For any $r\in (0,\alpha-1)$, it holds that
\begin{displaymath}
\sup_{k\geq 1} \E\big[(I_k)^r \!\mid\! I_k\geq 1\big]<\infty.
\end{displaymath}
\end{lemma}

\begin{proof}
First, given any fixed integer $k_0\geq 1$, there exists some constant $c>0$ such that 
\begin{displaymath}
\P(I_k\geq 1)\geq c  \quad \mbox{ for every } k\in \{1,\ldots,k_0\}.
\end{displaymath} 
As $r<\alpha-1$ it is clear that 
$\E[(\widehat N)^r ]= \sum_{k\geq 1} k^{1+r} \rho(k)<\infty$, and then
\begin{displaymath}
\max_{1\leq k\leq k_0} \E\big[(I_k)^r \!\mid\! I_k\geq 1\big]\leq \frac{1}{c}\E\big[(I_k)^r\big]\leq \frac{1}{c}\E\big[(\widehat N)^r \big]<\infty.
\end{displaymath}

For notational convenience, we define a random variable $I_k^*$ under $\P$ having the conditional distribution of $I_k$ given $I_k\geq 1$. Since
\begin{displaymath}
\E\big[(I_k^*)^r\big]\leq \sum_{n\geq 1} r n^{r-1} \P(I_k^*\geq n),
\end{displaymath}
it suffices to show that for any $r\in (0,\alpha-1)$, we can find a constant $C>0$ such that for all $k$ sufficiently large and for all $x\geq 1$,
\begin{equation}
\label{eq:Ik-tail}
\P(I_k^*\geq x) \leq \frac{C}{x^r}.
\end{equation}
To this end, we rewrite \eqref{eq:Ik-gene} as
\begin{displaymath}
\int_0^\infty \P(I_k^*\geq x) \theta e^{-\theta x} \,\mathrm{d}x = \alpha (1-e^{-\theta})^{\alpha-1} \frac{(q_{k-1})^{\alpha-1} \widetilde{L}(q_{k-1}(1-e^{-\theta}))}{\P(I_k\geq 1)}, 
\end{displaymath}
and recall that 
\begin{displaymath} 
\frac{(q_{k-1})^{\al-1}\widetilde L(q_{k-1})}{\P(I_k\geq 1)} \build{\longrightarrow}_{k \to\infty}^{} \,\frac{1}{\al}\,.
\end{displaymath} 
Using Potter's bounds for $\widetilde L$ and the previous convergence, we know that for any $\delta\in(0,\alpha-1)$, there exists $k_0(\delta)\in \N$ only depending on $\delta$ such that for every $k\geq k_0(\delta)$, 
\begin{displaymath}
\frac{\widetilde L(q_{k-1}(1-e^{-\theta}))}{\widetilde L(q_{k-1})}\leq 2 (1-e^{-\theta})^{-\delta} \quad \mbox{ for every } \theta>0,
\end{displaymath}
and 
\begin{displaymath}
\frac{(q_{k-1})^{\al-1}\widetilde L(q_{k-1})}{\P(I_k\geq 1)} \leq \frac{2}{\alpha}.
\end{displaymath}
It follows that for every integer $k\geq k_0(\delta)$, 
\begin{displaymath}
\int_0^\infty \P(I_k^*\geq x) \theta e^{-\theta x} \,\mathrm{d}x  \leq 4(1-e^{-\theta})^{\alpha-1-\delta},
\end{displaymath}
from which one can deduce \eqref{eq:Ik-tail} using the classical arguments for Karamata's Tauberian theorem and the monotone density theorem. For more details, we refer the reader to~\cite[Section 1.7]{BGT87}. 
\end{proof}

The next result generalizes Lemma 17 in~\cite{LIN2}.
\begin{lemma}
\label{lemma:5rv-cv-in-law}
We have
\begin{displaymath}
\Big(\frac{L_{k+1}}{M_k},\frac{L_k}{M_{k-1}},M_kc_k,M_{k-1}c_{k-1},M_{k-1}h_{k-1}\Big)  \, \build{\longrightarrow}_{k\to\infty}^{(\mathrm{d})}\, \big(\mathcal{R}_\al,\mathcal{R}'_\al,\textstyle \sum \mathcal{C}\a,\textstyle \sum'\mathcal{C}\a, \widehat{\mathcal{C}}\a \, \big),
\end{displaymath}
where in the limit:
\begin{itemize}
\item $\mathcal{R}_\al$ and $\mathcal{R}'_\al$ are two positive random variables with the same distribution given by
    $$\P(\mathcal{R}_\al>x)=(1+x)^{-\frac{\al}{\al-1}}  \quad \mbox{for all } x\geq 0\,;$$
\item Take $\widehat N_\al$ and $\widehat N'_\al$ two integer-valued random variables following the size-biased distribution of~$\theta_\al$. 
Let $(\cc\a_k)_{k\geq 2}$ and $(\widetilde\cc\a_k)_{k\geq 2}$ be two sequences of random variables identically distributed according to $\gamma_\al$. 
Then
$$\textstyle \sum \mathcal{C}\a \colonequals \cc\a_2+\cdots+\cc\a_{\widehat N_\al}  \quad \mbox{and} \,\quad \textstyle \sum' \mathcal{C}\a \colonequals \widetilde\cc\a_2+\cdots+ \widetilde\cc\a_{\widehat N'_\al}\,;$$
\item $\widehat{\mathcal{C}}\a$ is distributed according to $\widehat \gamma_\al$.
\end{itemize}
Furthermore, we suppose that all the random variables listed above in the description of the limit are defined under the probability measure $\P$, and they are independent.
\end{lemma}

\begin{proof}
We only give an outline of the proof, which is divided into three steps. 
First of all, using \eqref{eq:epsilon-estimate} one can verify that 
\begin{equation}
\label{eq:ML-conv}
\Big(\frac{L_k}{M_{k-1}},\frac{L_{k+1}}{M_k}\Big)  \, \build{\longrightarrow}_{k\to\infty}^{(\mathrm{d})}\, \big(\mathcal{R}_\al',\mathcal{R}_\al \big)
\end{equation}
holds for a similar reason as display (32) in~\cite{LIN2}.

Then notice that conditionally on $M_k$ and on the number $I_{M_k}$ of subtrees grafted at $\mathbf{u}_{M_k}$ that reach generation 0, these subtrees are independent $\rho$-Galton--Watson trees conditioned to have height greater than $M_k-1$. 
Recall that $c_k$ is the sum of the conductances between $\mathbf{u}_{M_k}$ and generation 0 in those subtrees. 
On account of Proposition~\ref{prop:cv-conductance} and Lemma~\ref{lem:nb-graft}, we can obtain the convergence in distribution of $M_k c_k$ to $\sum \mathcal{C}\a$ by calculating the Laplace transform of $M_k c_k$. 
A slight change of this argument by first conditioning on $M_{k-1}$ gives the joint convergence 
\begin{equation}
\label{eq:c-conv}
\big(M_{k-1}c_{k-1},M_kc_k\big)  \, \build{\longrightarrow}_{k\to\infty}^{(\mathrm{d})}\, \big(\textstyle \sum' \mathcal{C}\a,\textstyle \sum \mathcal{C}\a\big),
\end{equation}
which holds jointly with~(\ref{eq:ML-conv}), provided we keep all the limiting random variables independent. 

Finally, recall that with the notation of Section~\ref{sec:cv-dis-cond}, $h_{k-1}=\mathcal{C}_{M_{k-1}-1}([\widecheck{\mathsf{T}}]^{M_{k-1}-1})$.
The convergence 
\begin{equation*}
M_{k-1}h_{k-1} \, \build{\longrightarrow}_{k\to\infty}^{(\mathrm{d})}\, \widehat{\mathcal{C}}\a
\end{equation*}
can be shown by the same approximation method used in the proof of Lemma 17 in~\cite{LIN2}. 
Since $h_{k-1}$ only depends on $M_{k-1}$ and the finite tree strictly above the vertex $\mathbf{u}_{M_{k-1}}$ in~$\widecheck{\mathsf{T}}$, the last convergence holds jointly with (\ref{eq:ML-conv}) and (\ref{eq:c-conv}), provided $\widehat{\mathcal{C}}\a$ is taken to be independent of all the random variables involved in the definition of $(\mathcal{R}_\al,\mathcal{R}'_\al,\sum \mathcal{C}\a,\sum' \mathcal{C}\a)$.
\end{proof}

Lemma 18 in \cite{LIN2} states that in the finite variance case ($\al=2$), the second moment of $M_k h_k$ is uniformly bounded. 
When $\al<2$, we have the following lemma. 

\begin{lemma}
\label{lemma:Lr-bdd-hk}
For every $r\in (0,\al)$, it holds that
$$\sup_{k\geq 1} \E\big[(M_k h_k)^r\big]<\infty.$$
\end{lemma}

\begin{proof}
We fix $r\in (0,\al)$. For any $\eta>0$, there exists a positive constant $C(\eta)$ sufficiently large so that $(a+b)^r\leq C(\eta)a^r+(1+\eta)b^r$ for every $a,b>0$. 
Using this inequality together with Lemma~\ref{lem:Lr-gw-condct}, we can easily adapt the proof of~\cite[Lemma 18]{LIN2} to show the existence of positive constants $C<\infty$ and $\xi<1$, both independent of $k$, such that for all $k\geq 2$,
$$\E\big[(M_kh_k)^r\big] \leq C+\xi \,\E\big[(M_{k-1}h_{k-1})^r\big].$$
The uniform boundedness of the sequence $(\E[(M_kh_k)^r])_{k\geq 1}$ thus follows. 
\end{proof}

Keeping the same notation as in Lemma \ref{lemma:5rv-cv-in-law}, we set
\begin{displaymath}
Q_{\infty}\colonequals \log \Bigg( 1+ \frac{\cc\a_2+\cdots+\cc\a_{\widehat N_\al}+\frac{1}{\mathcal{R}_\al}}{\frac{1+\mathcal{R}'_\al}{\mathcal{R}'_\al}}- \frac{1}{1+\mathcal{R}'_\al \big(\widehat{\mathcal{C}}\a+ \widetilde \cc\a_2+\cdots+\widetilde \cc\a_{\widehat N'_\al} \big)} \Bigg).
\end{displaymath}
Recall that 
\begin{displaymath}
Q_k=\log \bigg(1+ \frac{M_kc_k+\frac{M_k}{L_{k+1}}}{\frac{M_k}{L_k}}- \frac{\frac{M_k}{L_k}}{\frac{M_k}{L_k}+\frac{M_k}{M_{k-1}}(M_{k-1}h_{k-1}+M_{k-1}c_{k-1})}\bigg).
\end{displaymath}
The next result is the analog of Lemma 19 in~\cite{LIN2}.

\begin{lemma}
\label{lemma:auxiliary-L1}
\begin{enumerate}
\item[(i)] $Q_k$ converges to $Q_\infty$ in $L^1$ and in particular $\lim_{k\to \infty}\E[Q_k]= \E[Q_{\infty}]$.
\item[(ii)] We have the equality $\E[Q_{\infty}] = \frac{\al-1}{\al}\lal$.
\item[(iii)] It holds that
\begin{displaymath}
\sup_{i,j\geq 1} \E\big[|Q_iQ_j|\big]<\infty \quad \mbox{ and } \quad \sup_{i,j\geq 1} \E\big[(Q_iQ_j)^2\big]<\infty.
\end{displaymath}
Moreover, for any integer $n\geq 2$, we have 
\begin{equation}
\label{eq:L2n-unif-bdd}
\sup_{i\geq 1} \E\big[(Q_i)^{2n}\big]<\infty\,.
\end{equation}
\end{enumerate}
\end{lemma}

\begin{proof}
(i) As in the proof of Lemma 19 in~\cite{LIN2}, we know that 
\begin{displaymath}
|Q_k|\leq \max \bigg\{\log \Big(1+M_kc_k+\frac{M_k}{L_{k+1}}\Big), \log\Big(1+\frac{M_{k-1}}{2L_k}\Big)\bigg\}.
\end{displaymath}
For any $r>0$, there exists a constant $A>0$ such that $\log(1+x)\leq A+x^r$ for every $x>0$.
It follows that
\begin{displaymath}
|Q_k|\leq A+\Big(M_kc_k+\frac{M_k}{L_{k+1}}\Big)^r+ \Big(\frac{M_{k-1}}{L_k}\Big)^r \leq A+(M_kc_k)^r+\Big(\frac{M_k}{L_{k+1}}\Big)^r+ \Big(\frac{M_{k-1}}{L_k}\Big)^r .
\end{displaymath}
Recall that conditionally on $M_k$ and $I_{M_k}$, the quantity $c_k$ is the sum of $I_{M_k}$ i.i.d.~conductances of reduced Galton--Watson trees having height $M_k$. By Lemma~\ref{lem:Lr-gw-condct}, the first moments $(\E[n\,\mathcal{C}_n(\mathsf{T}^{*n})])_{n\geq 1}$ are uniformly bounded by a finite constant $K>1$. If $r<1$, we apply Jensen's inequality to see that
\begin{displaymath}
\E\big[(M_kc_k)^r\big]= \E\big[\E[(M_kc_k)^r\!\mid\! M_k, I_{M_k}]\big]\leq \E\big[\E[M_kc_k\!\mid\! M_k, I_{M_k}]^r\big]\leq K\,\E\big[(I_{M_k})^r\big].
\end{displaymath}
Then Lemma \ref{lemma:I_k-frac-moments} implies that for any $r\in (0,\alpha-1)$,
\begin{displaymath}
\sup_{k\geq 1} \,\E\big[(M_kc_k)^r\big]<\infty.
\end{displaymath}
As in the proof of assertion (i) in~\cite[Lemma 19]{LIN2}, one can also show that for any $r\in(0,\alpha-1)$, 
\begin{displaymath}
\sup_{k\geq 1} \,\E\Big[\Big(\frac{M_k}{L_{k+1}}\Big)^r\Big]<\infty.
\end{displaymath}
Hence, $(Q_k)_{k\geq 2}$ is bounded in $L^p$ for some $p>1$. But owing to Lemma~\ref{lemma:5rv-cv-in-law}, $Q_k$ converges in distribution to $Q_\infty$. Therefore, the uniform integrability of $(Q_k)_{k\geq 2}$ gives the convergence in~$L^1$.

(ii) For the calculation of $\E[Q_{\infty}]$, we set 
$$ V_\al \colonequals \frac{\mathcal{R}_\al}{1+\mathcal{R}_\al}\quad \mbox{ and }\quad V'_\al \colonequals \frac{\mathcal{R}'_\al}{1+\mathcal{R}'_\al}.$$ 
They are independent with the same law of density $\frac{\al}{\al-1}(1-x)^{\frac{1}{\al-1}}$ on $[0,1]$. 
Notice that
\begin{eqnarray*}
Q_\infty &=& \log \Bigg( V'_\al\Big(\cc\a_2+\cdots+\cc\a_{\widehat N_\al}+\frac{1}{\mathcal{R}_\al} \Big) +\frac{V'_\al\big(\widehat{\mathcal{C}}\a+ \widetilde \cc\a_2+\cdots+\widetilde \cc\a_{\widehat N'_\al}\big)}{1-V'_\al+V'_\al \big(\widehat{\mathcal{C}}\a+ \widetilde \cc\a_2+\cdots+\widetilde \cc\a_{\widehat N'_\al}\big)}\Bigg)\\
&=& \log (V'_\al) +\log\Bigg( \cc\a_2+\cdots+\cc\a_{\widehat N_\al}+\frac{1}{\mathcal{R}_\al}  +\bigg(V'_\al+\frac{1-V'_\al}{\widehat{\mathcal{C}}\a+ \widetilde \cc\a_2+\cdots+\widetilde\cc\a_{\widehat N'_\al}}\bigg)^{-1}\Bigg),
\end{eqnarray*}
and hence
$$ \E[Q_\infty]=\E[\log (V'_\al)] +\E\Bigg[\log\Bigg( \cc\a_2+\cdots+\cc\a_{\widehat N_\al}+\frac{1}{\mathcal{R}_\al}  +\bigg(V'_\al+\frac{1-V'_\al}{\widehat{\mathcal{C}}\a+ \widetilde \cc\a_2+\cdots+\widetilde\cc\a_{\widehat N'_\al}}\bigg)^{-1}\Bigg)\Bigg]. $$
Using (\ref{eq:c*-rde}) to rewrite the second term in the right-hand side, we obtain 
$$ \E[Q_\infty]=\E[\log (V'_\al)] +\E\bigg[\log\Big( \cc\a_2+\cdots+\cc\a_{\widehat N_\al}+\frac{1}{\mathcal{R}_\al}  +\widehat{\mathcal{C}}\a \Big)\bigg]. $$
However, the relation $\mathcal{R}_\al=\frac{V_\al}{1-V_\al}$ gives 
\begin{displaymath}
\E\bigg[\log\Big( \cc\a_2+\cdots+\cc\a_{\widehat N_\al}+\frac{1}{\mathcal{R}_\al}  +\widehat{\mathcal{C}}\a \Big)\bigg]=\E\Big[\log\Big( 1-V_\al+V_\al\big(\cc\a_2+\cdots+\cc\a_{\widehat N_\al}+\widehat{\mathcal{C}}\a \big)\Big)\Big]-\E[\log(V_\al)].
\end{displaymath}
As $\E[\log(V_\al)]=\E[\log (V'_\al)]> -\infty$, we deduce that 
$$ \E[Q_\infty]= \E\Big[\log\Big( 1-V_\al+V_\al\big(\cc\a_2+\cdots+\cc\a_{\widehat N_\al}+\widehat{\mathcal{C}}\a \big)\Big)\Big].$$
Now we apply (\ref{eq:c*-rde}) again to see that $\log(\widehat{\mathcal{C}}\a)$ has the same distribution as 
$$\log \Big(\widehat{\mathcal{C}}\a +\cc\a_2+\cdots+\cc\a_{\widehat N_\al} \Big)-\log\Big(1-V_\al+V_\al\big(\cc\a_2+\cdots+\cc\a_{\widehat N_\al}+\widehat{\mathcal{C}}\a \big)\Big).$$
Accordingly, we conclude by \eqref{eq:lal-2} that 
$$\E[Q_\infty]= \E\Bigg[\log\frac{\widehat{\mathcal{C}}\a +\cc\a_2+\cdots+\cc\a_{\widehat N_\al} }{\widehat{\mathcal{C}}\a}\Bigg]= \frac{\al-1}{\al}\lal\,.$$

(iii) The last assertion can be analogously shown as assertion (iii) in~\cite[Lemma~19]{LIN2}, by adapting the arguments there in the same way as we have done for assertion (i). Concerning the supplementary result \eqref{eq:L2n-unif-bdd}, it can be similarly treated since one can adjust the power exponent $r$ in the inequality $\log(1+x)\leq A+x^r$ to be arbitrarily small. We therefore omit the details. 
\end{proof}

We are now ready to finish the proof of Lemma~\ref{lem:sum-L2-conv}.
Assertions (i) and (ii) of Lemma~\ref{lemma:auxiliary-L1} give 
\begin{equation*}
\lim_{k\to \infty}\E\bigg[\frac{1}{k}\sum\limits_{j=2}^k Q_j\bigg]= \lim_{k\to \infty}\frac{1}{k}\sum\limits_{j=2}^k \E[Q_j] =\E[Q_{\infty}]=\frac{\al-1}{\al}\lal.
\end{equation*}
With all the ingredients prepared above, we can prove the inequality 
\begin{equation*}
\limsup_{k\to \infty} \E\bigg[\Big(\frac{1}{k}\sum\limits_{j=2}^k Q_j\Big)^2\bigg]\leq \big(\E[Q_{\infty}]\big)^2
\end{equation*}
in a similar manner as Lemma~20 in~\cite{LIN2}. 
In the corresponding arguments, notice that for controlling the error of approximation, although Lemma~18 in~\cite{LIN2} is now replaced by its weaker analog Lemma~\ref{lemma:Lr-bdd-hk}, we are able to compensate by \eqref{eq:L2n-unif-bdd} to apply the H\"older inequality as required.  
Finally, the desired $L^2$-convergence follows from 
\begin{displaymath}
\limsup_{k\to\infty} \E\bigg[\Big(\frac{1}{k}\sum\limits_{j=2}^k Q_j - \E[Q_{\infty}] \Big)^2\bigg] \!\leq \limsup_{k\to\infty} \E\bigg[\Big(\frac{1}{k}\sum\limits_{j=2}^k Q_j\Big)^2\bigg]-2 \E[Q_{\infty}] \lim_{k\to\infty}\E\bigg[\frac{1}{k}\sum\limits_{j=2}^k Q_j\bigg]+ \E[Q_{\infty}]^2 \leq 0.
\end{displaymath}

\end{document}